\documentclass[11pt]{amsart}

\usepackage[top=1 in, bottom=1 in, left=1 in, right = 1 in]{geometry}
\usepackage{amsmath,enumerate}
\usepackage{amsthm}
\usepackage{amssymb}
\usepackage{setspace}
\usepackage{graphicx,float,subfloat,epstopdf,subfig}
\usepackage[colorlinks=true,urlcolor=blue,citecolor=OliveGreen]{hyperref}
\graphicspath{ {images/} }
\usepackage[dvipsnames]{xcolor}
\usepackage{algorithm}
\usepackage{algpseudocode}
\usepackage[font=small,labelfont=bf,width=0.9\textwidth]{caption}

\newcounter{count}
\stepcounter{count}

\newcommand{\abs}[1]{\left| #1  \right|}

\newcommand{\B}[1]{{\boldsymbol #1}}
\newcommand{\R}{\mathbb{R}}

\newcommand{\mS}{\mathcal S}

\newcommand{\eps}{\varepsilon}

\newcommand*{\defeq}{\mathrel{\vcenter{\baselineskip0.5ex \lineskiplimit0pt
                     \hbox{\scriptsize.}\hbox{\scriptsize.}}}%
     		     =}

\def\XXint#1#2#3{{\setbox0=\hbox{$#1{#2#3}{\int}$ }
\vcenter{\hbox{$#2#3$ }}\kern-.6\wd0}}

\newcommand{\bF}{\boldsymbol{F}}

\newcommand{\bS}{\boldsymbol{\sigma}}

\newtheorem{theorem}{Theorem}[section] 

\newtheorem{lemma}{Lemma}[section]
\newtheorem{corollary}{Corollary}[section]

\newtheorem{remark}{Remark}[section]

\newtheorem{proposition}{Proposition}[section]
\newcommand{\tred}[1]{{\color{black}{#1}}}

\title{A Fokker-Planck Framework for Control of Epidemics}
\author{Christian Parkinson and Souvik Roy}
\date{\today}

\begin{document}

\maketitle

\begin{abstract}
We present a control framework for stochastic compartmental models in epidemiology. In this framework, rather than directly controlling the stochastic system, we perform optimal control of an associated Fokker-Planck equation, with the goal of steering the distribution of possible solutions of the stochastic system to some desirable state. In particular, this allows for robust control mechanism with uncertainty not only in the dynamics, but also in the initial data. We formulate and fully analyze a partial differential equation constrained optimization problem, including a proof of existence of optimal controls via analysis of the control-to-state map, and a characterization of optimal controls via the Pontryagin minimum principle. We describe the application of the sequential quadratic Hamiltonian method to our problem, which provides numerical approximations of optimal control maps. We demonstrate our method using a minimal stochastic susceptible-infected-recovered model with different choices of cost functionals that represent different policy-maker concerns.  
\end{abstract}

\section{Introduction}

Since the COVID-19 pandemic, optimal public policy for the control of infectious disease spread has become a problem of acute interest. In this manuscript, we propose a novel framework for control of compartmental epidemiological models based on a partial differential equation (PDE) constrained optimization problem involving the Fokker-Planck equation associated with a stochastic differential equation modeling spread of a generic infectious disease. \tred{This work serves to fill a gap in the literature. While the Fokker-Planck control formulation has been successfully used in models from biology and the applied sciences, it has not been applied to mathematical epidemiology. We contend that the framework is particularly apt for mathematical epidemiology for a few reasons detailed below.}

Several authors studied general optimal control of basic epidemiological models even before the onset of the COVID-19 pandemic \cite{OCgen1,OCgen2,OCgen3}, and since the pandemic, there have been more targeted studies regarding various facets of epidemic control such as quarantine implementation \cite{OCQ1,OCQ2,OCQ3}, vaccination strategies \cite{OCvax1,OCvax2,OCvax3}, spatial heterogeneity \cite{OCspatial1,OCspatial2,OCspatial3}, and human behavioral effects \cite{OCbehavior1,OCbehavior2,OCbehavior3}. Several of these resources consider control of systems involving stochastic effects, but since they rely on classical results such as the (stochastic) Pontryagin maximum principle \cite[Ch. 2.5]{FlemingRishel}, \cite[Ch. 3.3]{YongZhou}, they require deterministic initial data. 

Parallel to this work in optimal control of epidemic models, there has been theoretical development of so-called Liouville \cite{Liou1} or Fokker-Planck \cite{FPCgen1,Borzi1} control frameworks, wherein one controls a system of deterministic or stochastic differential equations with uncertain initial data, with the goal of achieving a desirable distribution of possible states at some future time. Among other applications, these have been successfully used in models of collective motion for both pedestrians \cite{FPmotion1,FPmotion2,FPmotion3} and traffic flow \cite{FPtraffic}, object identification in image analysis \cite{FPimageanalysis}, and optimal medical treatment \cite{FPbio1,FPbio2,FPbio3}.    

While the Fokker-Planck control framework has not yet been applied to epidemiological models, we argue that it is very natural for this application for a few reasons. First, uncertainty in the total infection size is virtually guaranteed in the early stages of an epidemic, and this framework allows for uncertain initial data in contrast to any framework based on classical methods such as dynamic programming or direct application of the Pontryagin minimum principle to the stochastic system. Second, this framework provides robust control strategies which consider the whole distribution of possible solutions to the system of stochastic differential equations that describe disease spread, rather than optimizing for an individual trajectory. This is especially advantageous in the presence of $L^1$ control costs where the optimal control strategy is often bang-bang (see \cite[Ch. 17]{BangBang1} or \cite{BangBang2,BangBang3}), meaning that nearby trajectories may have significantly different optimal controls. For such systems, direct open loop control of the stochastic system is precarious because random perturbations may push the system to nearby states where the open loop controller provides severely suboptimal control actions. Third, control of the distribution of possible trajectories provides natural answers to probabilistic questions that may be of interest for policy-makers. For example, assuming a certain stochastic model, suppose a policy-maker wants to know whether a given control strategy will result in a 50\% likelihood of disease extinction by a given future time. This question is more difficult to answer if one considers control of individual trajectories, but very easy to answer if one controls the distribution of all possible trajectories.  Thus, to close this gap in the literature, we describe and analyze a simple open loop Fokker-Planck control framework applied to a minimal stochastic model for epidemiology, though with small tweaks, this framework can be applied much more broadly. \tred{In particular, while inspired by the recent COVID-19 pandemic, this work applies to generic disease spread and is purposefully agnostic to the particularities of any single epidemic. }

\tred{Beyond the mathematical analysis, our results illustrate the manner in which different policy-maker priorities lead to different control strategies. In different real-world scenarios, goals of the policy-maker may vary from disease annihilation to ``flattening the curve" to preserving a large susceptible population while awaiting vaccine availability. These, and any number of other public-health objectives, are easily modeled in this framework, and our results demonstrate the manner in which each of these can be at least partially accomplished while also accounting for the cost tradeoff they entail. For example, if the policy-maker assesses cost based purely on disease prevalence, a combination of non-pharmaceutical interventions and increased treatment efforts are employed. By contrast, if the goal is to ``flatten the curve" to avoid overwhelming the healthcare system, a much stronger use of non-pharmaceutical interventions supplemented by vaccination is the optimal strategy. Thus beyond allowing the policy-maker to evaluate control strategies in the presence of both initial and dynamic stochasticity, Fokker-Planck control provides a natural mechanism for evaluating control strategies and probabilistic public-health outcomes under varying policy-maker concerns. }

Besides formulation and analysis of the Fokker-Planck control framework for epidemiology, we also introduce the sequential quadratic Hamiltonian (SQH) method \cite{SQH} in the context of optimal control of PDE-based epidemiological models for the first time (though it has been used for ODE-based models of epidemiology in \cite{SQHepi1} and in the recent preprint \cite{OCbehavior2}). This is a numerical method for approximating optimal control maps which is equipped with a much more robust theoretical framework than some other popular methods such as the forward-backward sweep method used by \cite{OCgen3, OCvax3, OCQ2, OCbehavior1} among many others. The SQH method is relatively easy to implement and is completely explainable, in contrast to popular software packages such as the freely available IPOPT\footnote{\url{https://github.com/coin-or/Ipopt}} \cite{IPOPT}, as used by \cite{OCQ1,OCvax2}, which is very robust and explainable to experts in optimization, but is often effectively a black box to practitioners.

The remainder of the manuscript is structured as follows. In section~\ref{sec:model}, we introduce a simple controlled stochastic susceptible-infected-recovered (SIR) model, make some remarks about the associated controlled Fokker-Planck equation, and formulate an optimal control problem. In section~\ref{sec:analysis}, we prove existence of optimal control maps for our problem via analysis of the control-to-state map and derive first-order optimality conditions which characterize the optimal controls. In section~\ref{sec:computation}, we discuss the numerical methods we use to computationally approximate the optimal controls, and in particular, give an exposition of the SQH method \cite{SQH} as it pertains to our problem. In section~\ref{sec:results}, we present results of the application of our method to control problems involving our model wherein costs are accumulated in different manners representing different potential concerns of policy-makers. Finally, we give some concluding remarks and discuss avenues of future work in section~\ref{sec:conclusion}. 

\section{A controlled stochastic SIR model and Fokker-Planck equation} \label{sec:model}

We begin from the basic SIR model of Kermack and McKendrick \cite{SIR_KM}, which models infectious disease spread through a homogeneous population which is split into susceptible, infected, and recovered subpopulations. We assume a natural birth rate $b$, natural death rate $\delta$, infection rate $\beta$ and recovery rate $\gamma$, and we add three controls which may be available to a real-world policy-maker: (1) a reduction in infectivity of the disease due to the implementation of non-pharmaceutical intervention (NPI) measures such as mask-wearing, shelter-at-home, or social-distancing mandates; (2) vaccination efforts which move susceptible individuals directly to the recovered class without passing through the infected class; (3) increase in recovery rate due to treatment efforts. Representing these controls (respectively) by $\alpha \in [0,\alpha_{\text{max}}], v \in [0,v_{\text{max}}], \eta \in [0,\eta_{\text{max}}]$, this leads to the model \begin{equation}
    \label{eq:controlledSIR}
    \begin{split}
        \dot S &= b-(1-\alpha(t))\beta SI - (v(t)+\delta)S, \\
        \dot I &= (1-\alpha(t))\beta SI - (\gamma+\eta(t)+\delta)I, \\
        \dot R &= (\gamma+\eta(t))I+v(t)S - \delta R,
    \end{split}
\end{equation} along with nonnegative initial conditions $S(0) = S_0, I(0) = I_0, R(0) = R_0$. The flow diagram for this model is included in figure~\ref{fig:flowDiagram}, where any arrows flowing out of a compartment denote flow proportional to the compartment they are leaving. In particular, the colored arrows in the diagram represent our modifications to the basic SIR model. These are the terms where the control variables appear. 

\begin{figure}[t!]
    \centering
    \includegraphics[width=0.7\linewidth]{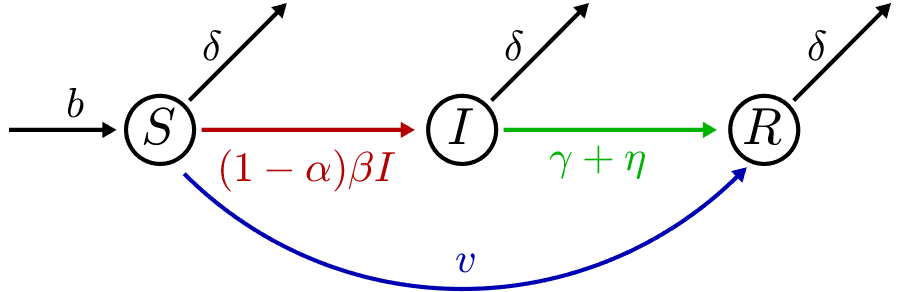}
    \caption{The flow diagram for \eqref{eq:controlledSIR} in the absence of noise. Colored lines represent our modifications to the basic SIR model; these are the terms where the control variables $\alpha,\eta,v$ appear.}
    \label{fig:flowDiagram}
\end{figure}

We add stochasticity into the model in two ways. First, we treat the initial data as uncertain, so that rather than being specified exactly, we assume $S_0,I_0,R_0$ are drawn from a specified probability distribution. And second, we perturb each equation in \eqref{eq:controlledSIR} with an It\^o noise term with respective strengths $\sigma_S,\sigma_I,\sigma_R$. In the epidemiological literature, common choices for these stochastic terms are to either let the noise in each equation be proportional to the population which the equation describes \cite{JYJS,Yamakazi,ZZY}, or to more explicitly model uncertainty in disease transmission by taking $\sigma_I = -\sigma_S \propto SI$ which is equivalent to formally replacing $\beta$ with $\beta+\sigma_\beta \dot W$ where $W$ is It\^o white noise \cite{Gray,JJS,PW3,Tornatore}. In the Fokker-Planck control literature, it is relatively common to let these terms be constant for ease of analysis; for example \cite{Borzi1,FPmotion3}. In any of these cases, $\sigma_S$ and $\sigma_I$ are independent of $R$, which means that $\dot S$ and $\dot I$ are entirely independent of $R$, at which point the $R$ compartment is simply a receptacle for those who are vaccinated or recovered, and has no bearing on disease spread dynamics. In accordance with this observation and in order to simplify the model, we will always make the assumption that $\sigma_S,\sigma_I$ are independent of $R$, whereupon we focus on the following system of two stochastic differential equations describing infectious disease spread: \begin{equation}
    \label{eq:SIR2d}
    \begin{split}
        dS &= \Big(b-(1-\alpha)\beta SI -(v(t)+\delta)S\Big)dt + \sigma_S(S,I,\alpha,\eta,v)dW_1,\\
        dI &= \Big((1-\alpha)\beta S I -(\gamma + \eta(t)+\delta)I\Big)dt + \sigma_I(S,I,\alpha,\eta,v)dW_2
    \end{split} 
\end{equation} where $W_1,W_2$ are the (stochastically independent) components of a 2-dimensional Wiener process. 

\tred{In the case of constant controls, typical analysis of epidemiological models like this entails derivation of stability analysis of \emph{disease free} and \emph{endemic} equilibria in terms of the so-called \emph{reproductive ratio}. For the deterministic model \eqref{eq:controlledSIR}, this is entirely classical. For the stochastic models it can be accomplished as in \cite{Gray,JJS,PW3,Tornatore}. Because the focus of this study is the optimal control of the Fokker-Planck equation detailed below, we omit the discussion of these concepts. For the sake of simulations in section \ref{sec:computation}, we simply choose $\beta,\gamma$ such that, in the absence of controls, a large outbreak of the disease will indeed occur.}

\tred{For more compact notation, we can define $X = (S,I)$ and $u =(u_1,u_2,u_3) \defeq \tred{(\alpha,\eta,v)}$, and write \eqref{eq:SIR2d} as \begin{equation}
    \label{eq:dynamics}
    dX = \bF(X,u)dt + \bS(X,u)d\B W
\end{equation} where \begin{equation} \label{eq:F}\B F(X,u) = \begin{bmatrix}
    b-(1-u_1)\beta SI - (u_2+\delta)S, \\
    (1-u_1)\beta SI - (\gamma+u_3+\delta)I.
\end{bmatrix}\end{equation} and \begin{equation} 
\label{eq:Sig} \bS(x,u) = \begin{bmatrix} \sigma_1(X,u) & 0 \\ 0 & \sigma_2(X,u) \end{bmatrix} = \begin{bmatrix}
    \sigma_S(S,I,\alpha,\eta,v) & 0 \\ 0 & \sigma_I(S,I,\alpha,\eta,v)
\end{bmatrix},
\end{equation} and $\B W$ is a 2-dimensional Weiner process with stochastically independent components. 

To describe the time evolution of the stochastic process governed by~\eqref{eq:dynamics}, we consider its probability density function (PDF). Let \( x = (x_1, x_2)^T \), and define \( f(x,t) \) as the PDF of \( X(t) \), representing the probability density of the system being in state \( x \) at time \( t \).} We assume the existence of a bounded domain $\Omega \subset \R^2$ with Lipschitz boundary which is positively invariant under the dynamics \eqref{eq:dynamics}, whereupon we consider initial data $X_0$ drawn from a probability distribution with density function $f_0(x)$ defined on $\Omega$. Note that existence of such $\Omega$ is not a prohibitive assumption. So long as the functions $\sigma_S,\sigma_I,\tfrac{\sigma_S}S, \tfrac{\sigma_I}{I}$ are bounded for bounded values of their inputs (which is true, for example, when $\sigma_S,\sigma_I$ are continuous and proportional to $S,I$ respectively), global existence of a bounded, nonnegative solution of \eqref{eq:dynamics} follows directly as in \cite{Gray,JJS,JYJS,ZZY}, whereupon one can take $\Omega = (a_{\text{min}},a_{\text{max}})^2$ for any $a_{\text{min}} \le 0$ and $a_{\text{max}}>0$ sufficiently large. Besides this, for the sake of theoretical results below, we will always assume that $\sigma_S^2,\sigma_I^2$ are twice continuously differentiable in $x$, and continuous in $u$. 

Given this, the Fokker-Planck equation describing the evolution of the probability distribution function for the solution of \eqref{eq:dynamics} is given by \begin{equation}\label{eq:FPcontrolled}
\begin{aligned}
&\partial_t f(x,t) + \nabla \cdot \big(\bF(x, u)\,f(x,t)\big) = 
\frac{1}{2}\sum_{j=1}^2 \frac{\partial^2}{\partial x_j^2}\big(\sigma_j^2(x,u)\,f(x,t)\big),\quad \mbox{in } \Omega \times(0,T],\\
&f(x,0) = f_0(x), \quad \mbox{in } \Omega,
\end{aligned}
\end{equation} along with zero flux boundary conditions $\mathcal F \cdot \hat n = 0$ on $\partial \Omega \times (0,T]$, where $\mathcal F$ is the flux, defined componentwise by \begin{equation} \label{eq:flux}
\mathcal{F}_j(x,t) = \frac 1 2\frac{\partial}{\partial x_j} \Big(\sigma_j(x,t)^2 f(x,t) \Big) - \bF_j(x,u) f(x,t),\quad j = 1,2.
\end{equation} For a derivation of the Fokker-Planck equation and discussion of related topics, see \cite{FPBook}. The time horizon $T > 0$ in \eqref{eq:FPcontrolled} is somewhat artificial. In theory, the time interval $(0,T)$ could be replaced with $(0,\infty)$. However, for the sake of formulating an optimal control problem, we will consider the finite time interval. 

Note that \eqref{eq:FPcontrolled} can be written in divergence form in which case, dropping arguments for brevity, the equation reads \begin{equation}\label{eq:FPdiv}
\partial_t f + \sum^2_{j=1} \Big(b_j\frac{\partial f}{\partial x_j}\Big) + cf  = \frac 1 2 \sum^2_{j=1} \frac{\partial}{\partial x_j}\Big(\sigma_j^2 \frac{\partial f}{\partial x_j} \Big)
\end{equation} where $b_j = F_j - \frac{\partial }{\partial x_j}\sigma_j^2$ and $c = \frac{\partial b_1}{\partial x_1} + \frac{\partial b_2}{\partial x_2}.$ This formulation is amenable to the definition of weak solutions as described in \cite[Ch. 7.1]{Evans},\cite[Ch. 3.5]{WYW}. Formally, we say that $f$ is a weak solution of \eqref{eq:FPcontrolled} if it satisfies the initial condition and {\small\begin{equation}
    \label{eq:weakFP}
    \int_{\Omega}\bigg( f_t \varphi + \frac 1 2 \sum^2_{j=1} b_j \frac{\partial f}{\partial x_j} \varphi + cf\varphi + \frac 1 2\sum^2_{j=1} \sigma^2_j \frac{\partial f}{\partial x_j}\frac{\partial \varphi }{\partial x_j}\bigg) dxdt = 0, \,\, \text{ for all } \varphi \in C_c^{1}(\Omega) \text{ and a.e.} \, t \in [0,T].
\end{equation} }

Since $\B F(x,u)$ is a smooth function of $x$ and remains bounded so long as $u$ is bounded and $\sigma_S,\sigma_I$ are twice continuously differentiable in $x$, and continuous in $u$, the coefficients $b_j$ and $c$ are in $L^\infty$. Thus the following proposition regarding existence and uniqueness of a weak solution of \eqref{eq:FPcontrolled} follows from \cite[Ch. 3, Theorem 3.5.1]{WYW}, \cite[Ch. IV, Theorem 9.1]{Ladyzhenskaja}.

\begin{proposition} \label{prop:FPexistence}
For fixed bounded control map $u:[0,T]\to  \R^3$ and $f_0 \in L^2(\Omega)$, there exists a unique weak solution of \eqref{eq:FPcontrolled} satisfying  $f  \in L^2(0,T;H^1(\Omega))$  and $f_t \in L^{2}(0,T;H^{-1}(\Omega))$. Moreover, if $f_0 \ge 0$ in $\Omega$ and $\|f_0\|_{L^1(\Omega)} = 1$, then $f\ge 0$ in $Q$ and $\|f(\cdot,t)\|_{L^1(\Omega)} = 1$ for all $t \in (0,T].$
\end{proposition}  

We henceforth define our solution space $$\mathcal W = \{f  \in L^2(0,T;H^1(\Omega)) \,\, : \,\, f_t  \in L^2(0,T;H^{-1}(\Omega))\}.$$ Since the coefficients $b_j,c,\sigma^2_j$ are all bounded, standard parabolic regularity results (e.g. \cite[Ch. 7, Thm. 2]{Evans}) give a bound on $\|f\|_{L^2(0,T;H^1(\Omega))}$ which is uniform over choices of controls $u$ taking values in a compact set. Given this, we consider controls taking values in the rectangle $R_U = [0,\alpha_{\text{max}}]\times [0,v_{\text{max}}] \times [0,\eta_{\text{max}}]$, with \tred{$\alpha_{\text{max}} \leq 1$},  which have total variation bounded by some constant $M>0$.  That is, the admissible control space is \begin{equation} \label{eq:admissibleControl}
 U_{\text{ad}} \defeq \Big\{u:[0,T] \to \R^3 \, \Big| \,u_i \in BV[0,T] \text{ with } TV(u_i) \le M, \,\, \text{ for } i=1,2,3 \Big\} \subset (L^\infty[0,T])^3.
\end{equation} We recall that the total variation $TV(v)$ of a function $v:[0,T]\to\R$ is defined by \begin{equation} \label{eq:TV}
    TV(v) = \sup_{\varphi \in C^1[0,T], \|\varphi\|_{L^\infty[0,T]} \le 1} \int^T_0 v(t) \varphi'(t)dt,
\end{equation} and the space $BV[0,T]$ is the set of all functions $v$ with $TV(v)<\infty$. We discuss this choice of control space shortly.

In a slight abuse of vocabulary, for a given $u \in U_{\text{ad}}$, we will often refer to the function $f$ guaranteed by proposition \ref{prop:FPexistence} as the solution of \eqref{eq:FPcontrolled} corresponding to $u$, brushing over the fact that we are considering a weak solution as defined above.

Our optimal control problem for \eqref{eq:FPcontrolled} is then stated abstractly as \begin{equation} \label{eq:optControl} \begin{split}
&\min_{f,u} J(f,u) \defeq \int^T_0 \ell(u(t)) dt + \int_\Omega K(x)f(x,T)dx+ \int^T_0 \int_\Omega G(x,t)f(x,t) dx dt  \\
&\text{such that } u\in U_{\text{ad}} \text{ and } f \text{ satisfies } \eqref{eq:FPcontrolled}, 
\end{split}\end{equation} for some functions $\ell,K,G$. That $J(f,u)$ is bounded below follows, as long as $K$ and $G$ are bounded, from the nonnegativity and $L^1$ boundedness of $f$, which is uniform over choices of $u \in U_{\text{ad}}$. Because of this, we will, at a minimum, assume that $K,G$ are bounded on their domains, and specify some further constraints as needed below. 

\tred{We have chosen the abstract cost functional \eqref{eq:optControl} to represent the preferences of a generic policy-maker. In applications, $\ell(u)$ represents the cost associated with implementing controls. This could represent cost of synthesizing and distributing vaccines, loss in GDP associated with implementing partial quarantines, excess spending on treatment efforts, or any number of other costs. The function $G: Q \to \R$ corresponds to running cost assigned to different states throughout the evolution. This would need to be chosen by the policy-maker based on their objectives and their assessment of what constitutes desirable outcomes. For example, one policy-maker may assign higher costs to states with higher number of infections, meaning that $G(x,t)$ is larger when $x_2 = I$ is larger, with the simple goal of lowering total number of infections. Alternatively, another policy-maker may have the goal of ``flattening the curve" so as to not overly burden the healthcare system. In this case, $G(x,t)$ may be zero for low values of $I$, but then discontinuously jump to some large value when $I$ crosses a threshold representing hospital capacity. This function would assign no cost to the mass of $f(x,t)$ which is below the threshold, but very high cost to any mass of $f(x,t)$ above the threshold. Likewise, $K:\Omega \to \R$ represents final cost, which assess cost based on the final distribution $f(x,T)$. This expresses the policy-maker's goals for distribution of states at the final time $T$. For example, if vaccination is in production but not yet available, it may be an express goal of the policy-maker to maintain a significantly large susceptible population by the time vaccines are deployed. In this case, one would like to steer the mass of $f(x,T)$ to states where $x_1 = S$ is large. Maximizing $S$ is equivalent to minimizing $-S$, so one may set $K(x)$ to be negative and decreasing in $S$ to accomplish this. All this to say, we choose an abstract functional because we would like to prove our results in the most generic case, but these functions directly express the policy-maker's preferences, and policy-makers who model these functions differently will have vastly different optimal intervention strategies.}

\begin{remark} \label{rem:controlspace}
    \emph{Our particular choice of control space $U_{\text{ad}}$ is quite important for three reasons. The first is a practical consideration: this control space allows for discontinuous control maps, which is especially important because for $L^1$ control problems, one needs to allow for the possibility of bang-bang controls. This gives this control space a natural advantage over, for example, $(H^1[0,T])^3$ which would only allow for continuous controls. The second is that $U_\text{ad}$ is closed under weak $L^2$ limits. This follows directly from the definition \eqref{eq:TV} since $v_k \rightharpoonup v$ in $L^2[0,T]$  means that inequalities for $\int^T_0 v_k(t)\varphi'(t)dt$ that are uniform over choices of  $k$ and $\varphi$ will be preserved in the limit. The third is that by the Helly selection theorem \cite[Thm. 5.5]{EvansMeasure}, any sequence in $U_{\text{ad}}$ admits a subsequence converging pointwise a.e. which allows us to pass limits into integrals using the dominated convergence theorem.}
\end{remark}

\section{Analysis of the Fokker-Planck control problem} \label{sec:analysis}

In this section, we provide some analysis of the Fokker-Planck control problem \eqref{eq:optControl}. Specifically, we give necessary conditions which characterize the optimal control maps via a version of the Pontryagin Minimum Principle (PMP) in the vein of \cite{Borzi1}. Before doing so, we prove existence of an optimal control by analyzing the control-to-state map. \tred{To streamline the exposition, some proofs have been moved to the appendix. This is true for any theorems or lemmas below which do not include a proof.}

\begin{theorem}
     \label{eq:opControlexistence}
    Assume that the control cost $\ell:R_U\to [0,\infty)$ is continuous and convex, and that $\sigma_S^2,\sigma_I^2$ are twice differentiable in $x$ and Lipschitz continuous in $u$ with a Lipschitz constant which is uniform over $x \in \Omega$. Then there exists an optimal control $u^* \in U_{\text{ad}}$ and corresponding solution $f^*$ of \eqref{eq:FPcontrolled} which solve \eqref{eq:optControl}.
\end{theorem}

\begin{proof} The proof follows the general strategy of \cite[Corollary 1]{Antil}, which entails analysis of the control-to-state map. Here the control-to-state map \begin{equation} \label{eq:controlToState} u\in U_{\text{ad}} \subset (L^2[0,T])^3  \,\,\,\,\, \mapsto \,\,\,\,\, \mathcal S(u) = f \in \mathcal W\subset L^2(Q), \text{ the solution of } \eqref{eq:FPcontrolled} \end{equation} is well-defined by virtue of proposition~\ref{prop:FPexistence}. Note, we are considering $U_{\text{ad}}$ with the norm $$\|u\|_{(L^2[0,T])^3} = \left(\sum^3_{i=1} \|u_i\|^2_{L^2[0,T]}\right)^{1/2}.$$ We prove the following two properties: \begin{itemize}
    \item[(a)] $\mathcal S$ is weakly closed as a map from $U_{\text{ad}}\to L^2(Q).$ That is, if a sequence $\{u^{(k)}\}$ in $U_{\text{ad}}$ is such that $u^{(k)} \rightharpoonup u$ in $(L^2[0,T])^3$ and $\mS(u^{(k)}) \rightharpoonup f \in L^2(Q)$, then $u \in U_{\text{ad}}$ and $\mS(u) = f.$
    \item[(b)] $\mathcal S$ is weak-strong continuous as a map from $U_{\text{ad}} \to L^2(Q).$
\end{itemize} For (a), take a sequence $\{u^{(k)}\}$ in $U_{\text{ad}}$ such that $u^{(k)} \rightharpoonup u$ in $(L^2[0,T])^3$ and $\mS(u^{(k)}) \rightharpoonup f$ in $L^2(Q)$. That $u \in U_{\text{ad}}$ follows from the reasoning in remark~\ref{rem:controlspace}. By the Helly selection theorem, we can pass to a subsequence of $\{u^{(k)}\}$ which converges pointwise a.e. to $u$; we do so without renaming the sequence. Now $\{\mS(u^{(k)})\}$ is a bounded sequence in the Hilbert space $\mathcal W$, so it has a subsequence converging weakly to some $f \in \mathcal W$. On the other hand, $\mathcal W$ is compactly embedded in $L^2(Q)$ by the Aubin-Lions lemma \cite[Ch. III, Prop. 1.3]{Showalter} and hence, passing to a further subsequence if necessary, we can attain strong convergence to $f$ in $L^2(Q)$. By the uniform $L^\infty$ bounds on the coefficients $b_j,c$ in \eqref{eq:weakFP}, weak convergence of $\mS(u^{(k)})$ to $f$ in $\mathcal W$, strong convergence of $\mS(u^{(k)})$ to $f$ in $L^2(Q)$, and pointwise a.e. convergence of $u^{(k)}\to u$ (which is required for application of the dominated convergence theorem) we see from \eqref{eq:weakFP}, that $f$ must be a weak solution of \eqref{eq:FPcontrolled} corresponding to the control $u \in U_{\text{ad}}.$ By definition, $\mS(u)$ is the unique such weak solution, so $\mS(u)=f.$ This concludes the proof of (a). 

For (b), we use a simple application  of the Urysohn subsequence principle \cite[Ch. 2.1.17]{TaoMeasure}. Suppose that $\{u^{(k)}\}$ is a sequence converging weakly to $u \in U_{\text{ad}}.$ Again $\{\mS(u^{(k)})\}$ is a uniformly bounded sequence in $\mathcal W.$ Thus any subsequence $\{\mS(u^{(k_\ell)})\}$ is uniformly bounded in $\mathcal W$, and so there is a further subsequence $\{\mS(u^{k_{\ell_m}})\}$ converging strongly to some $f \in L^2(Q).$ By the closure guaranteed in (a), we must have $S(u)=f.$ Thus every subsequence of $\{\mS(u^{(k)})\}$ has a further subsequence converging strongly to $\mS(u)$, so Urysohn's principle implies that $\mS(u^{(k)}) \to \mS(u)$ and $\mS$ is weak-strong continuous from $U_{\text{ad}}\to L^2(Q)$. 

Finally, consider the reduced cost functional $\mathcal J(u) = J(\mS(u),u).$ Since the control-to-state operator is well-defined, problem \eqref{eq:optControl} is equivalent to \begin{equation}
    \label{eq:reducedControlProb} \min_{u \in U_{\text{ad}}} \mathcal J(u)
\end{equation} Now $\mathcal J(u)$ is bounded from below so $\inf_{u \in U_{\text{ad}}} \mathcal J(u) = J^*$ is finite. Take a minimizing sequence $\{u^{(k)}\}$ such that $\mathcal J(u^{(k)}) \to J^*.$ Since $\{u^{(k)}\}$ is a bounded sequence in $(L^2[0,T])^3$, we can pass to a subsequence which converges weakly to some $u^* \in U_{\text{ad}}$, and by weak-strong continuity proven in (b), we know that $\mS(u^{(k)}) \to \mS(u)$ strongly in $L^2(Q)$ along this subsequence. Now, $$\mathcal J(u^{(k)}) = \int^T_0 \ell(u^{(k)}(t))dt + \int_\Omega K(x)[\mS(u^{(k)})](x,T)dt + \int^T\int_\Omega G(x,t) [\mS(u^{(k)})](x,t)dxdt.$$ The conditions that $\ell$ is nonnegative, continuous and convex guarantee that the first term above is weakly lower-semicontinuous in $u$. The weak-strong continuity of $\mS(u)$ guarantees that the latter two terms are well-behaved in the limit as $k \to \infty$ since $K,G$ are bounded. Thus we see $$J^* \le \mathcal J(u^*) \le \liminf_{k\to\infty} \mathcal J(u^{(k)}) = J^*,$$ so indeed this $u^* \in U_{\text{ad}}$ and the corresponding $f^* \defeq \mS(u^*) \in \mathcal W$ are a solution of \eqref{eq:optControl}.
\end{proof}

We remark that we also have strong-strong continuity of the control-to-state map. While this is not enough to conclude existence of a solution to the optimal control problem, it does have bearing on the convergence of the numerical methods. 

\begin{lemma}
    \label{lem:strongstrong}
    The control to state map $\mS: U_{\text{ad}} \to \mathcal W$ is Lipschitz continuous.
\end{lemma}

In order to characterize this optimal pair, for functions $f,q:Q \to \R$ and controls $u \in U_{\text{ad}}$, we formally define the Lagrangian for our problem as \begin{equation} \label{eq:lagrangian1}
    L(f,q,u) = J(f,u) + \int^T_0 \int_\Omega q\Big(\partial_t f +\nabla\cdot\big(\B F(x,u) f\big) - \frac 1 2 \sum^2_{j=1} \frac{\partial^2}{\partial x_j^2}(\sigma_j(x,u)^2 f) \Big) dxdt
\end{equation} where we have suppressed the dependence of $f,q$ on $(x,t)$ and the dependence of $u$ on $t$ for brevity.

Writing out $J(f,u)$ and integrating by parts, we see \begin{equation} \label{eq:lagrangian2} \begin{split}
    L(f,q,u) &=  \int^T_0 \bigg(\ell(u) + \int_{\Omega} f\Big(-\partial_tq - \bF(x,u) \cdot \nabla q - \frac 1 2 \sum_{j=1}^2 \sigma_j(x,u)^2\frac{\partial^2q}{\partial x_j^2} +G(x,t) \Big)dx \bigg)\\
    &\hspace{1cm} \int_\Omega [K(x)+q(x,T)]f(x,T) - f_0(x)q(x,0) dx.
    \end{split}
\end{equation}

This motivates the definition of the adjoint equation of \eqref{eq:FPcontrolled} for our problem:   \begin{equation}\label{eq:adjFPcontrolled} \scriptstyle
\begin{aligned}
&-\partial_t q(x,t) -  \bF(x, u)\,\cdot \nabla q(x,t) - 
\frac{1}{2}\sum_{j=1}^2 \sigma_j^2(x,u)\frac{\partial^2q}{\partial x_j^2}(x,t) + G(x,t) = 0, \quad \mbox{in } \Omega \times(0,T)\\
&q(x,T) = -K(x), \quad \mbox{in } \Omega,
\end{aligned}
\end{equation} along with a zero flux boundary condition $(D\nabla q) \cdot \hat n = 0$ on $\partial \Omega \times (0,T)$ where $D$ is a diagonal matrix with entries $\sigma_j^2$.

Again, existence, uniqueness and some regularity for \eqref{eq:adjFPcontrolled} follows from \cite[Ch. IV, Theorem 9.1]{Ladyzhenskaja} given some mild assumptions on $G(x,t),K(x)$ (in addition to the assumptions we have already imposed on $\sigma_S,\sigma_I$).

\begin{proposition}
Assume that $K \in L^2(\Omega)$ and $G \in L^2(Q)$. Then for a fixed control map $u \in U_{\text{ad}}$, there exists a unique solution $q \in L^2(0,T;H^1(\Omega))$ of \eqref{eq:adjFPcontrolled}.
\end{proposition}

Using similar formality as with the Lagrangian \eqref{eq:lagrangian1}, we now define the Pontryagin Hamiltonian function {\small\begin{equation}
    \label{eq:Hamiltonian} 
    H(t,f(\cdot,t),q(\cdot,t),u(t)) = \ell(u(t)) - \int_\Omega \Big( f(\cdot,t)\B F(\cdot,u(t))\cdot \nabla q(\cdot,t) + \frac 1 2 \sum^2_{j=1} f(\cdot,t) \sigma_j(\cdot,u(t))^2\frac{\partial^2q}{\partial x_j^2}(\cdot,t) \Big)dx
\end{equation}} for $f,q:Q\to\R$ and $u \in U_{\text{ad}}.$ Here we have suppressed the dependence of $f,q$ on $x$ but maintain the dependence on $t$ to emphasize that this Hamiltonian does depend on time.  

We want to prove that the optimal controls can be characterized as a pointwise (in time) minimizer of this Hamiltonian, rather than in terms of a variational inequality which is typical for general controlled partial differential equations; see \cite[Ch. 9]{Manzoni}, \cite[Ch. 5]{Tro}. To this end, we prove a lemma expressing the difference of the cost functional evaluated on different controls explicitly in terms of the Hamiltonian \eqref{eq:Hamiltonian}. 

\begin{lemma} \label{lem:hamiltonianRelationship}
    For any two control maps $u, \overline u \in U_{\text{ad}}$ with respective solutions $f$ and $\overline f$ of \eqref{eq:FPcontrolled}, we have $$J(f,u) - J(\overline f,\overline u) = \int^T_0 \Big[H(t,\overline f(\cdot,t),q(\cdot,t),u(t)) - H(t,\overline f(\cdot,t),q(\cdot,t),\overline u(t)) \Big]dt$$ where $q$ is the solution of \eqref{eq:adjFPcontrolled} corresponding to control $u$.  
\end{lemma}

With this, we prove a version of the PMP for the Fokker-Planck control problem \eqref{eq:optControl}.

\begin{theorem}
    Let $(f^*,u^*)$ be a solution of \eqref{eq:optControl}, and let $q^*$ be the solution of the corresponding adjoint equation \eqref{eq:adjFPcontrolled}. Then \begin{equation}
        \label{eq:PMP}
        H(t,f^*(\cdot,t), q^*(\cdot,t),u^*(t)) = \min_{w \in R_U} H(t,f^*(\cdot,t), q^*(\cdot,t),w).
    \end{equation} for almost all $t \in [0,T].$
\end{theorem}

\begin{proof}
    The proof uses a so-called needle variation of the optimal control map $u^* \in U_{\text{ad}}$. Taking any admissible control action $w \in R_U$ and any $t_0 \in [0,T]$, we let $B_k(t_0)$ be a sequence of balls centered at $t_0$ such that $\tred{\abs{B_k(t_0)}} \to 0$ as $k\to \infty$, and define \begin{equation}
    \label{eq:needleVariation}
    u_k(t) = \begin{cases}
        u^*(t), & t \in [0,T] \setminus B_k(t_0), \\ w, & t \in [0,T] \cap B_k(t_0).
    \end{cases}
\end{equation} We let $f_k$ be the solution of \eqref{eq:FPcontrolled} and $q_k$ be the solution of \eqref{eq:adjFPcontrolled} corresponding to $u_k$. Then by lemma \ref{lem:hamiltonianRelationship}, we see  \begin{equation} \label{eq:thm11}
    \begin{split}
       0 \le \frac{J(f_k,u_k) - J(f^*,u^*)}{\abs{B_k(t_0)}} &= \frac{1}{\abs{B_k(t_0)}}\int^T_0 \Big[H(t,f^*, q_k,u_k) - H(t,f^*, q_k, u^*)\Big]dx\,dt\\
       &= \frac{1}{\abs{B_k(t_0)}} \int_{B_k(t_0)}  \Big[H(t,f^*, q_k,w) - H(t,f^*, q_k, u^*)\Big] dt \\
       &= \frac{1}{\abs{B_k(t_0)}} \int_{B_k(t_0)} \Big[H(t,f^*, q^*,w) - H(t,f^*, q^*, u^*)\Big] dt  \\
       &\hspace{1cm}+\frac{1}{\abs{B_k(t_0)}} \int_{B_k(t_0)}\int_\Omega f^* \B F(x,u^*)\cdot \nabla(q_k -q^*) \,dx\,dt\\
       &\hspace{1cm}+ \frac{1}{\abs{B_k(t_0)}}\int_{B_k(t_0)}\int_\Omega \frac 1 2\sum^2_{j=1}f^*\sigma_j(x,u^*)^2\frac{\partial^2}{\partial x_j^2}(q_k -q^*) \, dx\,dt \\
       &\hspace{1cm}+\frac{1}{\abs{B_k(t_0)}} \int_{B_k(t_0)}\int_\Omega f^* \B F(x,w)\cdot \nabla(q_k -q^*) \,dx\,dt\\
       &\hspace{1cm}+ \frac{1}{\abs{B_k(t_0)}}\int_{B_k(t_0)}\int_\Omega \frac 1 2\sum^2_{j=1}f^*\sigma_j(x,w)^2\frac{\partial^2}{\partial x_j^2}(q_k -q^*) \, dx\,dt.
    \end{split}
\end{equation} After integrating by parts to pass a derivative onto $f^* \in \mathcal W$ and invoking the smoothness of $\B F(x,u), \bS(x,u)$ as functions of $x$, we see that the last four lines of \eqref{eq:thm11} go to zero as $k\to\infty$ by the strong continuity of the control-to-state operator proven in Lemma \ref{lem:strongstrong}. 

Finally, since $f^*,q^*$ are sufficiently regular, the map $t \mapsto [H(t,f^*, q^*,w) - H(t,f^*, q^*, u^*)]$ is in $L^1[0,T]$ and thus by the Lebesgue differentiation theorem \cite[Theorem 7.10]{Rudin}, we can take the limit as $k\to\infty$ in \eqref{eq:thm11} and conclude that for almost every $t \in [0,T]$,  $$0 \le H(t,f^*(\cdot,t),q^*(\cdot,t),w) - H(t,f^*(\cdot,t),q^*(\cdot,t),u^*(t)).$$ Since $w \in R_U$ was arbitrary, this proves \eqref{eq:PMP}.
\end{proof} 

Having proven this, we state the full necessary optimality conditions as a corollary.

\begin{corollary}[Necessary Conditions for Optimality] \label{cor31}
    If $(f,u)$ is a solution of the Fokker-Planck control problem \eqref{eq:optControl}, then there exists $q$ such that the following optimality system holds: \begin{equation}  
        \label{eq:forwardFP} \tag{FP}
        \left\{\begin{split}
            &\partial_t f(x,t) + \nabla \cdot \big(\bF(x, u)\,f(x,t)\big) = 
\frac{1}{2}\sum_{j=1}^2 \frac{\partial^2}{\partial x_j^2}\big(\sigma_j^2(x,u)\,f(x,t)\big),\quad \mbox{in } \Omega \times(0,T],\\
&f(x,0) = f_0(x), \quad \mbox{in } \Omega, \,\,\,\,\,\,\, \mathcal F \cdot \hat n = 0, \quad \mbox{on } \partial \Omega \, \text{ for } \mathcal F  \text{ as defined in \eqref{eq:flux}},
        \end{split}\right.
    \end{equation} 
    \begin{equation}  
        \label{eq:adjFP} \tag{ADJ}
        \left\{\begin{split}
            -&\partial_t q(x,t) -  \bF(x, u)\,\cdot \nabla q(x,t) = 
\frac{1}{2}\sum_{j=1}^2 \sigma_j^2(x,u)\frac{\partial^2q}{\partial x_j^2}(x,t) - G(x,t), \quad \mbox{in } \Omega \times[0,T)\\
&q(x,T) = -K(x), \quad \mbox{in } \Omega, \,\,\,\,\,\,\, \nabla q\cdot \hat n = 0, \quad \mbox{on } \partial \Omega,
        \end{split}\right.
    \end{equation} 
    \begin{equation}
        \label{eq:PMP2} \tag{PMP}
        H(t,f(\cdot,t), q(\cdot,t),u(t)) = \min_{w \in R_U} H(t,f(\cdot,t), q(\cdot,t),w). \hspace{4.5cm}
    \end{equation}
\end{corollary}

\tred{For some final notes, the authors of \cite{Borzi1} point out an interesting connection between \eqref{eq:adjFP} and the classical Hamilton-Jacobi-Bellman equation \cite[Ch. VI, \S4]{FlemingRishel} arising from application of the dynamic programming principle directly to control of \eqref{eq:dynamics}. In short, \eqref{eq:adjFP} is the Hamilton-Jacobi-Bellman equation corresponding to the dynamics \eqref{eq:dynamics} given an appropriately defined cost functional. Lastly, corollary~\ref{cor31} is analogous to the classical Pontryagin Maximum/Minimum Principle (PMP), as well as its stochastic counterpart. Generally, when analyzing control of a specific type of equation, the PMP involves the forward state equation, a similar backward adjoint equation, and an optimality condition. This applies in the classical setting or control of ordinary or stochastic differential equations \cite{FlemingRishel}, control of Vlasov-McKean dynamics \cite{carmonaFBSDE}, and in control of PDEs as seen here. In particular, since our state equation is a forward diffusive PDE, our adjoint equation is a backward diffusive PDE. } 

\section{Computational methods} \label{sec:computation}

In this section, we describe the numerical methods used to solve the FP control problem \eqref{eq:optControl}. This includes finite difference schemes for solving the forward and adjoint FP equations, \eqref{eq:forwardFP} and \eqref{eq:adjFP} respectively, and a brief exposition of the sequential quadratic Hamiltonian (SQH) method for resolving the optimal control maps. A fuller exposition of the SQH method with several examples can be found in \cite{SQH}.

To solve for the forward and adjoint Liouville equations \eqref{eq:forwardFP} and \eqref{eq:adjFP}, we use a strong stability preserving Runge-Kutta time discretization scheme coupled with the Chang-Cooper spatial discretization scheme and also implement a Strang splitting method. The details of the complete scheme can be found in \cite{Roy2024_2,FPbio2}.

The sequential quadratic Hamiltonian (SQH) method is an iterative method of resolving the optimal control for \eqref{eq:optControl} by using successive quadratic perturbations, which act as approximations to the true Hamiltonian \eqref{eq:Hamiltonian}. Pseudocode for the method is detailed in algorithm~\ref{algSQHmethod}, and we describe it here. For a penalization parameter $\eps>0$, we formally define the quadratically augmented Hamiltonian by \begin{equation} \label{eq:augmented_Hamiltonian}
    H_\eps(t,f(\cdot,t),q(\cdot,t),u(t),\tilde u(t)) = \tred{H(t,f(\cdot,t),q(\cdot,t), u(t))} + \eps\abs{u(t)-\tilde u(t)}^2_2
\end{equation} where $H$ is the Hamiltonian defined in \eqref{eq:Hamiltonian}, $f,q$ are functions on $Q$ and $u,\tilde u \in U_{\text{ad}}.$  We will use this augmented Hamiltonian while adaptively adjusting the parameter $\eps$ at each iteration of the SQH process. Specifically, $\eps$ is increased if a sufficient decrease in the functional $J(f,u)$ is not observed, and decreased if $J(f,u)$ decreases adequately. Here, $\tilde{u}$ represents the previous approximations of the control $u$ and the purpose of the quadratic term $\varepsilon \abs{u - \tilde{u}}_2^2 $ is to ensure that the pointwise minimizer of $H_\eps$, and thus the updates to $u$ remain close to the prior values $\tilde{u}$, especially when $\eps$ is large. An important note is that during the iteration, specifically in step (ii), the values of $f$ and $q$ used to update $u$ are those obtained from the previous iteration.

%
% NOTE: CHANGED eta TO mu IN THE ALGORITHM SINCE WE 
% HAVE USED eta FOR ONE OF THE CONTROL VARIABLES
%
\begin{algorithm}[t!]
\caption{Sequential Quadratic Hamiltonian Method for Resolving Optimal Control Maps}

\begin{algorithmic}
\State (1) Input:  initial approximation $u^0$, maximum number of iterations $k_{max}$, tolerance $\kappa >0$, 
$\varepsilon >0$, \\ \hspace{0.6cm} $\lambda >1$, $\mu > 0$, and $\zeta\in\left(0,1\right)$. \\

\State (2) Set $k = 0$ and compute the solution $f^0$ to the FP equation \eqref{eq:forwardFP} with control $u=u^0$. \\

\State (3) Perform the following iteration.\\

\Repeat\\
    \State (i) Compute the solution $q^k$ to the adjoint equation \eqref{eq:adjFP} with $f=f^k$ and $u = u^k$.\\

   \State (ii) Compute $u^{k+1}$ satisfying the following for a.e. $t \in [0,T]$\, :
    \tred{$$H_\eps \left(t, f^k(\cdot,t), q^k(\cdot,t), u^{k+1}(t), u^{k}(t)\right) = \min_{w \in R_U}\, H_\eps \left(t, f^k(\cdot,t), q^k(\cdot,t), w,u^{k}(t)\right)$$ }   

    \State (iii) Compute the solution $f^{k+1}$ to the Fokker-Planck equation \eqref{eq:forwardFP} with control $u=u^{k+1}.$\\
    
    \State (iv) Compute $\tau=\|u^{k+1 } -u^{k}\|^2_{(L^{2}[0,T])^3}$.\\

    \State (v) {\bf if:} $\left\{J\left(f^{k+1},u^{k+1}\right)-J\left(f^{k},u^{k}\right)> -\mu \, \tau\right\}$ 
    then set $\eps \leftarrow \lambda\eps$ and return to \tred{step (ii)},\\
    \hspace{\algorithmicindent}\hspace{0.6cm}{\bf else if:} 
    $\left\{J\left(f^{k+1},u^{k+1}\right)-J\left(f^{k},u^{k}\right) \leq -
    \mu \, \tau \right\}$, then 
    set $\eps \leftarrow \zeta \eps$ and continue. \\

    \State (vi) Set $k\leftarrow k+1.$

\Until{$k = k_{\text{max}}$ or $\tau < \kappa$}
\end{algorithmic}
\label{algSQHmethod}
\end{algorithm}

We further note that in step (v) of this algorithm, if the inequality $J\left(f^{k+1},u^{k+1}\right)-J\left(f^{k},u^{k}\right)> -\mu \, \tau$ holds, this indicates that a sufficient decrease in the objective functional $J(f,u)$ has not been achieved. In such a case, the setting $\eps \leftarrow \lambda \eps$ causes an increase in $\eps$ since $\lambda > 1$, and by returning to step (2), we restart the iteration with the updated augmented Hamiltonian function. Conversely, if the inequality does not hold, it confirms that the required reduction in $J(f,u)$ has been met. The updated control $u^{k+1}$ is then adopted, along with the corresponding updates $f^{k+1}$ and $q^{k+1}$ as the solutions to \eqref{eq:forwardFP} and \eqref{eq:adjFP} respectively. In this situation, $\eps \leftarrow \zeta \eps$ causes a reduction in $\eps$ since $\zeta \in (0,1).$

The following theorem gives the convergence result of the SQH algorithm for solving \eqref{eq:optControl}
\begin{theorem} \label{theoremSQH} Let $\left(f^{k},u^{k}\right)$ and $\left(f^{k+1},u^{k+1}\right)$ be generated by the SQH method (algorithm \ref{algSQHmethod}) applied to \eqref{eq:optControl}, with $u^{k+1}, u^k \in U_{ad}$. Then for the current value of $\eps > 0$ chosen by algorithm \ref{algSQHmethod}, the following inequality holds:
\begin{equation} \label{eq:decreaseJ} J\left(f^{k+1},u^{k+1}\right) - J\left(f^{k},u^{k}\right) \leq -\epsilon \, \| u^{k+1} - u^k \|^{2}_{(L^{2}[0,T])^3}.\end{equation}
In particular, this implies $J\left(f^{k+1},u^{k+1}\right) - J\left(f^{k},u^{k}\right) \leq  - \mu \, \tau$ for $\epsilon =\mu$ and $\tau = \| u^{k+1} - u^k \|^{2}_{(L^{2}[0,T])^3}$. \end{theorem}

\begin{proof}
In step (iii) of algorithm \ref{algSQHmethod}, for almost all $t\in [0,T]$, we have: $$ H_\eps(t, f^k, q^k, u^{k}, u^{k+1}) \leq H_\eps(t, f^k, q^k, u^{k}, w),
$$
for all $w \in R_U$. This implies $$ H_\eps(t, f^k, q^k, u^{k}, u^{k+1}) \leq H_\eps(t, f^k, q^k, u^{k}, u^k) = H(t, f^k, q^k, u^{k}). $$ Thus, we obtain the inequality
$$H(t, f^k, q^k, u^{k+1}) + \epsilon \, | u^{k+1} - u^k |^2_2 \leq H(t, f^k, q^k, u^{k}),$$
and applying lemma~\ref{lem:hamiltonianRelationship}, we have 
$$
\begin{aligned} 
J(f^{k+1}, u^{k+1}) - J(f^k, u^k) & = \int_0^{T} \Big(H(t, f^k, q^k, u^{k+1}) - H(t, f^k, q^k, u^{k})\Big) dt \\
& \le -\epsilon\| u^{k+1} - u^k \|^{2}_{(L^{2}[0,T])^3},
\end{aligned}
$$ as desired. 
 \end{proof}  

 As a consequence of theorem~\ref{theoremSQH}, we have the following result guaranteeing that algorithm~\ref{algSQHmethod} will terminate. 

 \begin{theorem}
\label{theoremSQHterminates}
If $\eps = \mu$ is chosen in algorithm~\ref{algSQHmethod} then 
\begin{itemize}
\item[(a)] the sequence $\{J(f^k,u^k)\}$ is decreasing and thus converges to some $J^* \ge \inf_{f,u} J(f,u)$, 
\item[(b)] $\displaystyle \lim_{k\to\infty} \|u^{k+1}-u^k\|_{(L^{2}[0,T])^3} = 0$
\end{itemize}
 \end{theorem}

 \begin{proof}
Statement (a) follows directly from \eqref{eq:decreaseJ}, and convergence follows since the sequence is decreasing and bounded below. 

For statement (b), we rearrange \eqref{eq:decreaseJ} to read $$\|u^{k+1}-u^k\|_{(L^{2}[0,T])^3}^2 \le \frac{1}{\mu}\Big(J(f^{k},u^k) - J(f^{k+1},u^{k+1})\Big)$$ whereupon we have the telescoping partial sum $$0 \le \sum^K_{k=0} \|u^{k+1}-u^k\|_{(L^{2}[0,T])^3}^2 \le \frac 1 \mu \Big(J(f^{0},u^0) - J(f^{K+1},u^{K+1})\Big).$$ Taking $K\to\infty$ shows that the infinite series converges, so the summands must tend to zero, achieving (b). 
 \end{proof}

In particular, this shows that in algorithm~\ref{algSQHmethod}, if one chooses $\eps = \mu$, then for any tolerance $\kappa>0$, the convergence criterion $\tau < \kappa$ will be reached in finitely many steps. Further, the Lipschitz continuity of the control-to-state map derived in lemma~\ref{lem:strongstrong} shows that $\|f^{k+1}-f^k\|_{L^2(Q)}$ will also tend to zero. 

\tred{\begin{theorem}[Convergence of SQH in the $L^2$--$L^1$ case]
Assume that
\[
\ell(u)=\frac{\beta_2}{2}|u|^2+\beta_1|u|_1,
\qquad
\beta_1,\beta_2>0,
\]
so that the Hamiltonian can be written in the form
\[
H(t,f_u,q_u,w)
=
\frac{\beta_2}{2}|w|^2
+\beta_1|w|_1
-r(w),
\]
where
\[
r(w) = \int_\Omega \Big( f_u(\cdot,t)\B F(\cdot,w)\cdot \nabla q_u(\cdot,t) + \frac 1 2 \sum^2_{j=1} f_u(\cdot,t) \sigma_j(\cdot,w)^2\frac{\partial^2q_u}{\partial x_j^2}(\cdot,t) \Big)dx.
\] Suppose that there exists a constant $C_r>0$ such that
\[
\|r(u)-r(\widetilde u)\|_{(L^2(0,T))^3}
\le
C_r
\|u-\widetilde u\|_{(L^2(0,T))^3}
\]
for all admissible controls $u,\widetilde u\in (L^2(0,T))^3$. If
\[
\beta_2>C_r,
\]
then, for every fixed $\varepsilon>0$, the SQH update defines a contraction mapping on
$(L^2(0,T))^3$.
Consequently, the sequence $\{u^k\}$ generated by Algorithm~1 converges strongly in
$(L^2(0,T))^3$ to a unique fixed point $u^*$.
Moreover,
\[
H(t,f^*(\cdot,t),q^*(\cdot,t),u^*(t))
=
\min_{w\in R_U}
H(t,f^*(\cdot,t),q^*(\cdot,t),w)
\]
for almost every $t\in[0,T]$, where
$f^*=S(u^*)$ and $q^*$ is the corresponding adjoint state.
\end{theorem}

\begin{proof}
For the augmented Hamiltonian
\[
H_\varepsilon
(t,f^k,q^k,w,u^k)
=
H(t,f^k,q^k,w)
+
\varepsilon |w-u^k|^2,
\]
the SQH update is given by
\[
u^{k+1}
=
\arg\min_{w\in R_U}
\left\{
\frac{\beta_2}{2}|w|^2
+\beta_1|w|_1
-r(w)
+\varepsilon |w-u^k|^2
\right\}.
\]
Since $R_U\subset\mathbb R_+^3$, we have
$|w|_1=\sum_{i=1}^3 w_i$.
Therefore, the minimization decouples componentwise and yields
\[
u^{k+1}
=
\Pi_{R_U}
\left(
\frac{
r(u^k)
+
2\varepsilon u^k
-
\beta_1\mathbf 1
}
{\beta_2+2\varepsilon}
\right),
\]
where $\Pi_{R_U}$ denotes the orthogonal projection onto $R_U$ and
$\mathbf 1=(1,1,1)^T$.
Define the SQH mapping
\[
T_\varepsilon(u)
=
\Pi_{R_U}
\left(
\frac{
r(u)
+
2\varepsilon u
-
\beta_1\mathbf 1
}
{\beta_2+2\varepsilon}
\right).
\]
Let $u_1,u_2\in (L^2(0,T))^3$.
Since the projection onto a closed convex set is nonexpansive,
\[
\|T_\varepsilon(u_1)-T_\varepsilon(u_2)\|
\le
\frac{1}{\beta_2+2\varepsilon}
\|
r(u_1)-r(u_2)
+
2\varepsilon (u_1-u_2)
\|.
\]
Using the Lipschitz continuity of $r$, we have
\[
\|T_\varepsilon(u_1)-T_\varepsilon(u_2)\|
\le
\frac{C_r+2\varepsilon}
{\beta_2+2\varepsilon}
\,
\|u_1-u_2\|.
\]
Since $\beta_2>C_r$,
\[
\frac{C_r+2\varepsilon}
{\beta_2+2\varepsilon}
<1.
\]
Hence, $T_\varepsilon$ is a contraction on
$(L^2(0,T))^3$.
By Banach's fixed-point theorem, there exists a unique fixed point
$u^*\in (L^2(0,T))^3$ and
\[
u^k\to u^*
\qquad
\text{strongly in }
(L^2(0,T))^3.
\]
At the fixed point,
\[
u^*
=
\arg\min_{w\in R_U}
H_\varepsilon(t,f^*,q^*,w,u^*).
\]
Since
\[
H_\varepsilon(t,f^*,q^*,u^*,u^*)
=
H(t,f^*,q^*,u^*),
\]
the fixed-point condition implies
\[
H(t,f^*,q^*,u^*)
\le
H(t,f^*,q^*,w)
\qquad
\forall\,w\in R_U,
\]
for almost every $t\in[0,T]$.
Therefore,
\[
H(t,f^*(\cdot,t),q^*(\cdot,t),u^*(t))
=
\min_{w\in R_U}
H(t,f^*(\cdot,t),q^*(\cdot,t),w)
\]
for almost every $t\in[0,T]$.
\end{proof}
}

\section{Results} \label{sec:results}
In this section, we present the results of our FP control framework \eqref{eq:optControl} simulated using the SQH algorithm. In all simulations we set $b = \delta = 0.01$. Having done so, solutions $S,I,R$ of the deterministic model \eqref{eq:controlledSIR} will satisfy $S+I+R \to 1$ as $t \to \infty$. Accordingly, if we choose $S_0 \in [0,1]$ and $I_0$ positive but close to zero, we empirically observe that $S(t),I(t) \in [0,1]$ for all $t$. Because of this, we choose the spatial domain of the FP equation \eqref{eq:FPcontrolled} to be $[0,1]\times [0,1]$. With infection rate $\beta = 3$ and recovery rate $\gamma = 1$, the disease dynamics have more-or-less entirely played out by time $T = 10$ (see figure~\ref{fig:uncontrolled}, top). The values of the upper bounds of the controls $\alpha$, $\eta$ and $v$ are set to $\alpha_{\text{max}} = 0.85$, $\eta_{\text{max}} = 0.25\gamma$ and $v_{\text{max}} = 0.1$, respectively denoting that the maximum decrease in infection rate due to non-pharmaceutical intervention is $85\%$, \tred{the maximum achievable increase in recovery rate due to treatment efforts is $25\%$}, and the maximum vaccination rate is $10\%$ of the population per unit time. For initial distribution of $(S_0,I_0)$, we consider a normal distribution centered at $(0.99, 0.01)$ with variance $0.025$ (we then multiply the distribution function by the indicator function of $[0,1]\times [0,1]$ and normalize in $L^1$). \tred{This represents a scenario where the policy-maker estimates that roughly 1\% of the population is infected at the outset, but there is some uncertainty in this estimate so that $I_0 \sim \mathcal N(0.01,0.025).$} The diffusion coefficients are chosen in the spirit of \cite{JJS,PW3}: $$\sigma_S = -\sigma_I = \sqrt{0.02} (1-\alpha)SI$$ which is meant to express uncertainty in the infection rate. We use a mixed $L^1 / L^2$ control cost $$\ell(u) = \beta_1 \|u\|_1 + \frac {\beta_2}2 \|u\|_2^2$$ where $\beta_1,\beta_2$ are constant. For our simulations we take $\beta_1 = 0.2, \beta_2 = 0.1$, though any nonnegative choices are acceptable. \tred{This represents a joint $L^1/L^2$ cost assessment. In application, true cost assessment may be more involved than a simple formula like this. Here $L^1$ cost is an abstract measure of ``total amount" of controls used, and is often preferred in biological applications as a more physically meaningful quantity, whereas $L^2$ costs have less physical meaning but better theoretical properties when considering the optimization problem. We note that the theory in section \ref{sec:analysis} is agnostic to the particular choice of $\ell$ as long as it is convex and lower-semicontinuous.} 

\tred{Henceforth we fix all of these parameters. We note that these are synthetic values, chosen to demonstrate different aspects of the model and not to represent any particular epidemic in any particular region.} 

\begin{figure}[t!]
\centering
\includegraphics[width=0.48\textwidth]{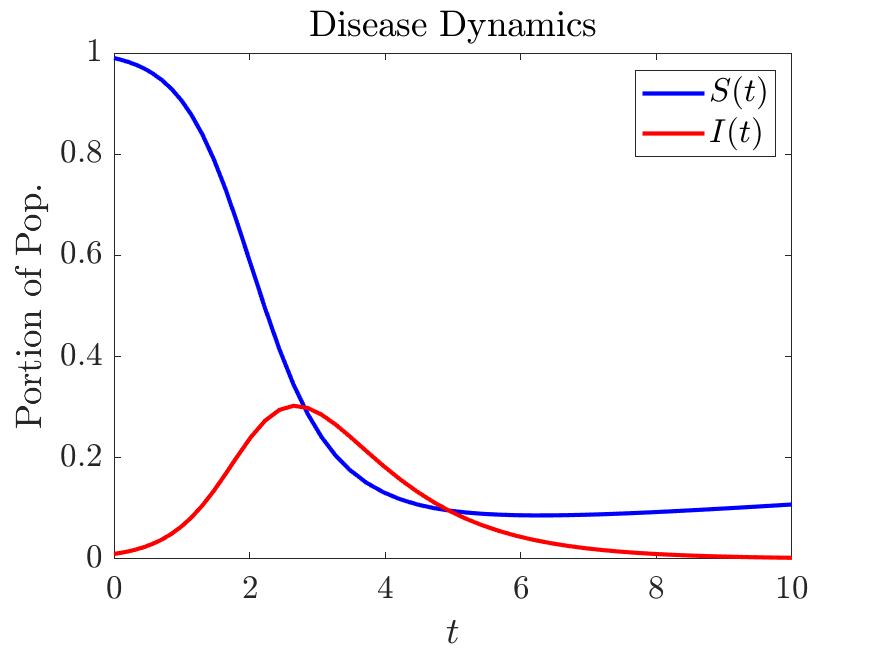} \,\,\, 
\includegraphics[width=0.48\textwidth]{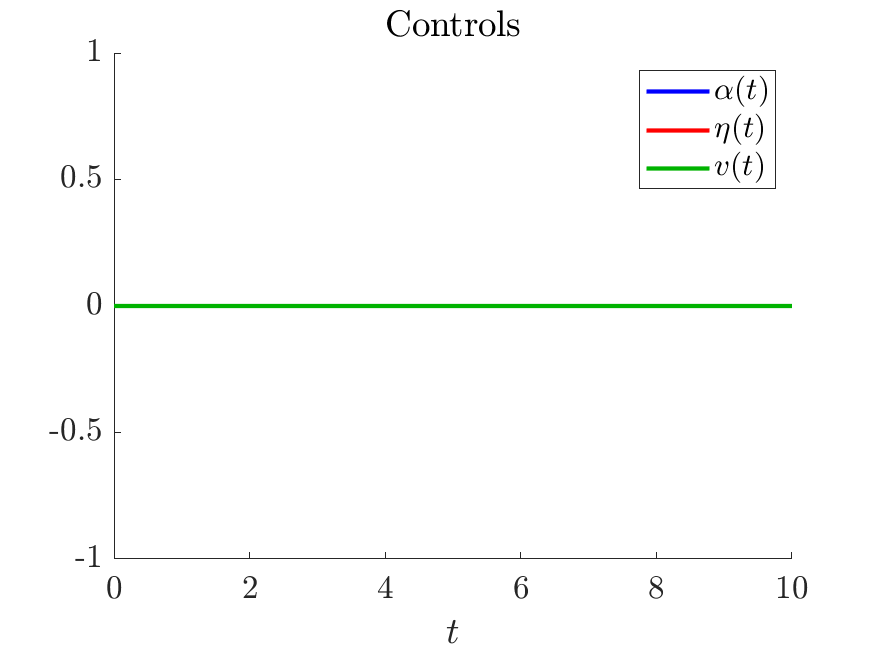}\\
\noindent\rule{0.9\textwidth}{1.5pt}
\includegraphics[width=0.3\textwidth,trim=55 0 75 0, clip]{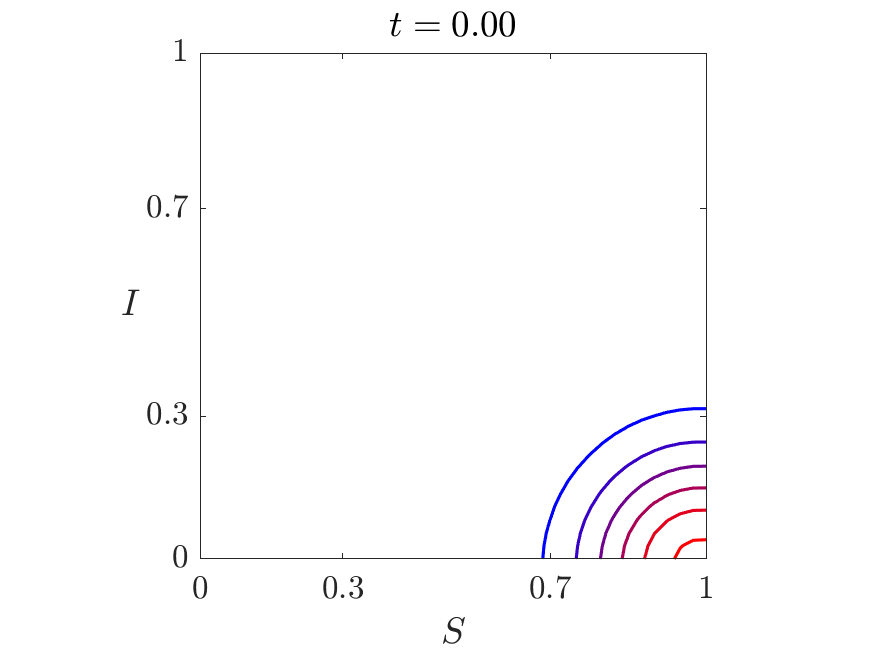} \, \includegraphics[width=0.3\textwidth,trim=55 0 75 0, clip]{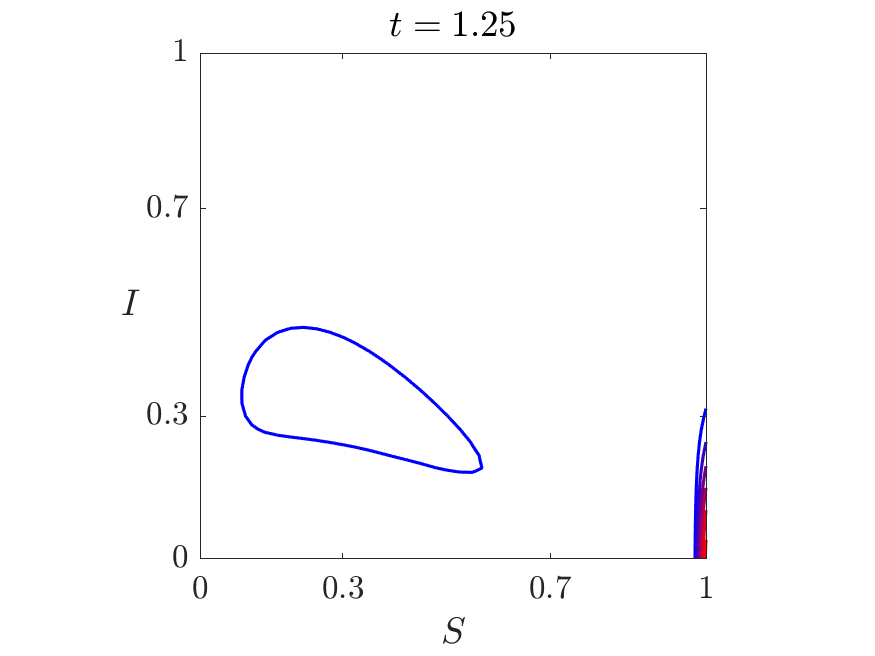} \,
\includegraphics[width=0.3\textwidth,trim=55 0 75 0, clip]{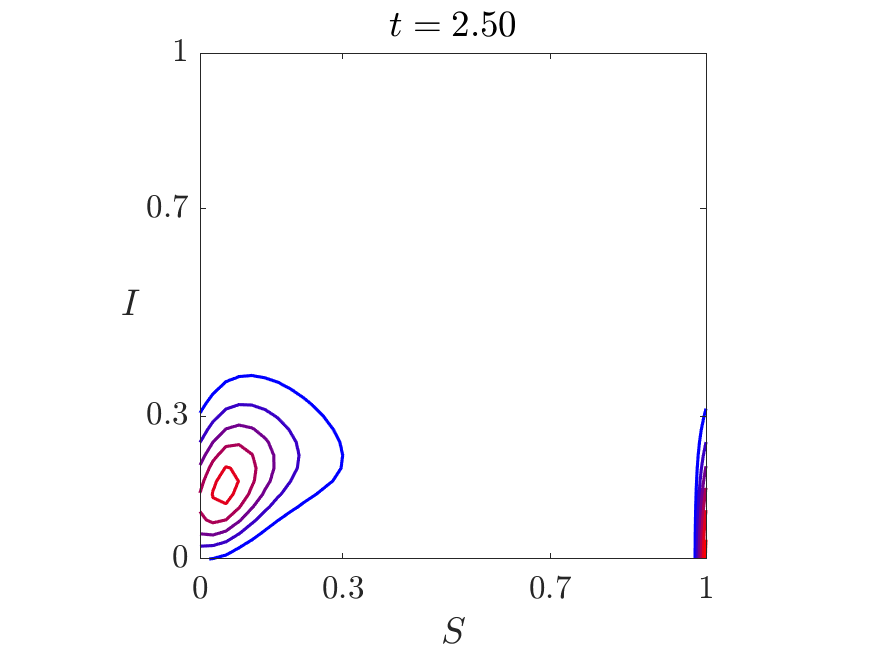} \\
\includegraphics[width=0.3\textwidth,trim=55 0 75 0, clip]{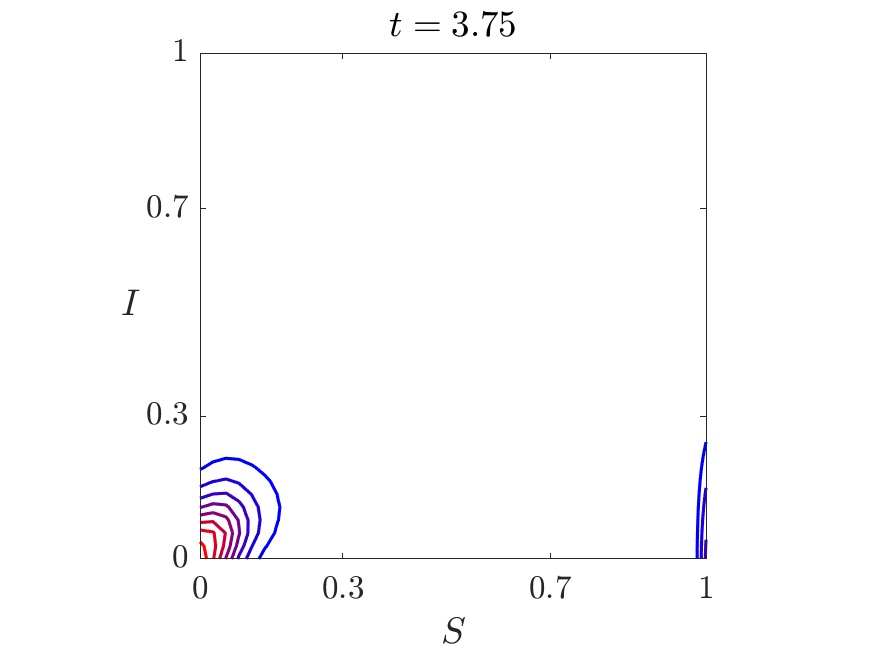} \,
\includegraphics[width=0.3\textwidth,trim=55 0 75 0, clip]{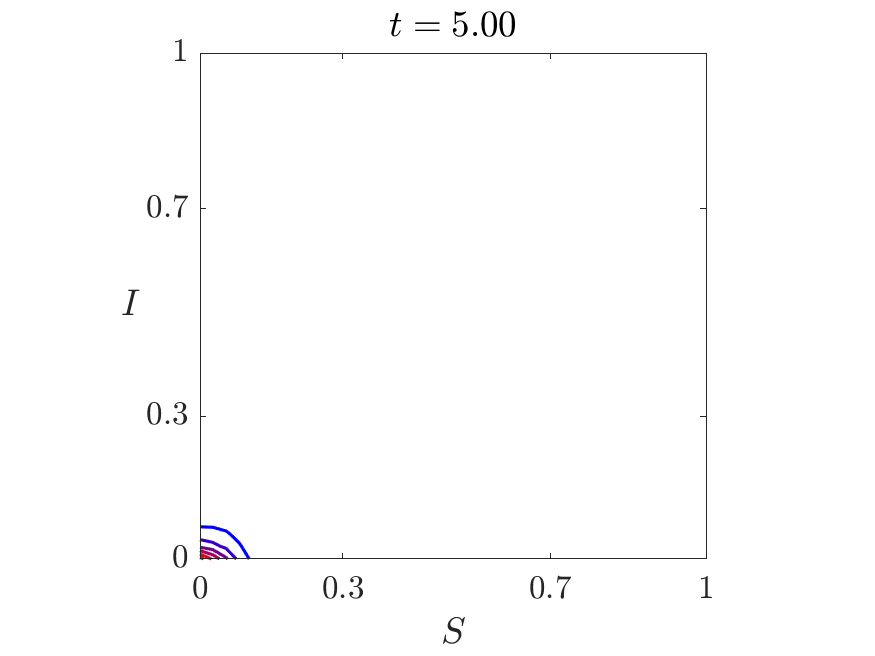} \,
\includegraphics[width=0.3\textwidth,trim=55 0 75 0, clip]{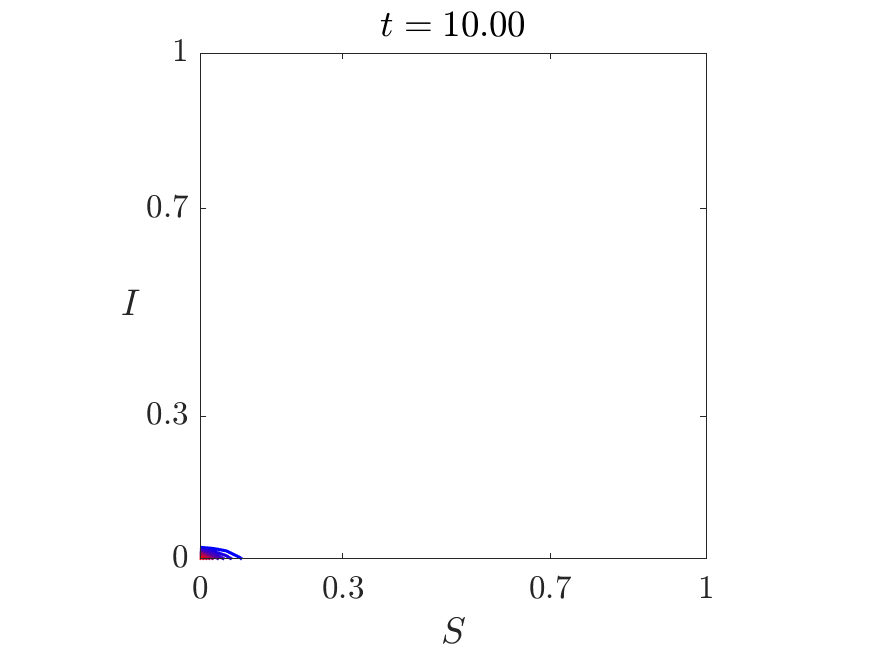} \,
\caption{The dynamics corresponding to our choice of parameters in the \emph{uncontrolled} case (i.e. $\alpha,\eta,v \equiv 0$). Top: the deterministic disease dynamics governed by \eqref{eq:controlledSIR}. Bottom: snapshots of the solution $f(x,t)$ of the Fokker-Planck equation \eqref{eq:FPcontrolled} in the form of contour plots with higher values in red and lower values in blue.}
\label{fig:uncontrolled}
\end{figure}

The initial guesses for the controls in the SQH algorithm are all set to 0. We discretize the spatial domain into 41 equally spaced points in each direction while the temporal grid consists of 81 equally spaced points. The hyperparameters for the SQH algorithm are set at $\mu = 10^{-9}, \zeta = 0.9, \lambda = 1.1, \kappa = 10^{-3}$ and $\varepsilon = 1.$ The maximum iteration count is $k_{\text{max}} = 150$ though this is never reached in our examples. 

Before solving any control problems, we include results of the simulation in the \emph{uncontrolled} scenario (i.e. $\alpha,\eta,v \equiv 0$) for comparison. These results are seen in figure~\ref{fig:uncontrolled}, where the top plots contain the deterministic disease dynamics following \eqref{eq:controlledSIR} starting from $(S_0,I_0) = (0.99, 0.01)$, and the bottom plots give snapshots of the solution $f(x,t)$ of the Fokker-Planck equation \eqref{eq:FPcontrolled} in the form of contour plots where red contours represent higher values and blue contours represent lower values. The dynamics show a peak infection of roughly $I = 0.3$ occurring at roughly $t = 3$. Referring to the distribution, we note that by time $t = 5$ (and certainly by time $t = 10$), essentially all the mass is centered near $(S,I) = (0,0)$, meaning that the infection has run its course, and a vast majority of the population is in the recovered class. There are some interesting artifacts in the distribution function toward the lower right hand corner of the plots when $t = 1.25, 2.50, 3.75$ which seem to indicate that there is a nonzero chance that the susceptible population remains prominent until those times. This could possibly be explained by the observations of Bertozzi et al. \cite{bertozzi2020challenges} regarding the relationship between the time scale for the disease and the initial infected population. In essence, they conclude that for simple compartmental models lowering the initial infection size $I_0$ will not affect the total infection size, but simply delay the onset of the wave of infections. The possibility of the susceptible population remaining large for such a long time could then be explained by the fact that our initial distribution allows for $I_0$ values that are arbitrarily small. 

Beyond this, we include three simulated scenarios involving optimal control. In each scenario, we choose different forms of the running cost $G(x,t)$ and terminal cost $K(x)$ from \eqref{eq:optControl}, corresponding to different concerns which a policy-maker may have. We recall that we are using variables $x = (x_1,x_2) = (S,I)$ here. \\

\begin{figure}[t!]
\centering
\includegraphics[width=0.48\textwidth]{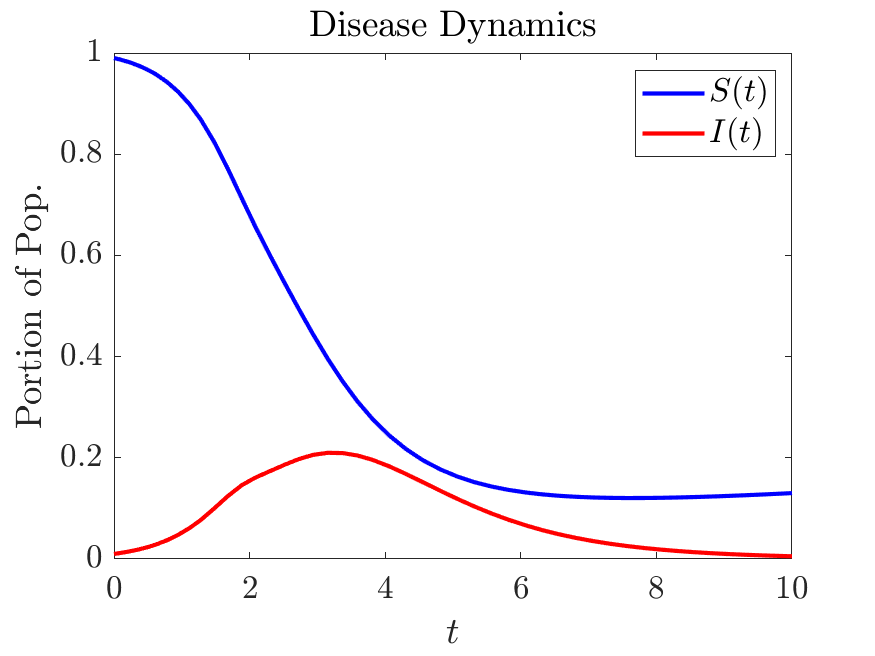} \,\,\, 
\includegraphics[width=0.48\textwidth]{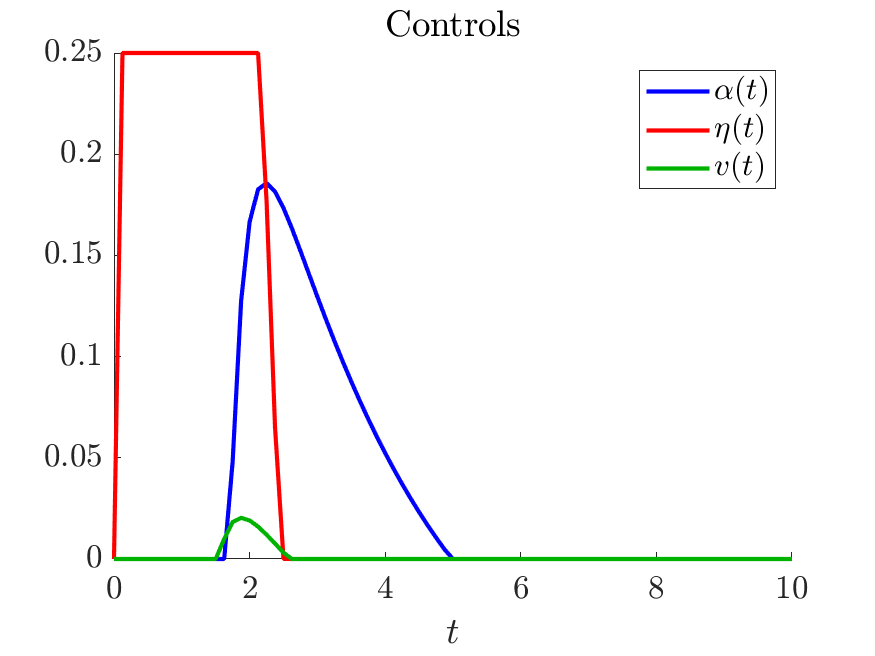}\\
\noindent\rule{0.9\textwidth}{1.5pt}
\includegraphics[width=0.3\textwidth,trim=55 0 75 0, clip]{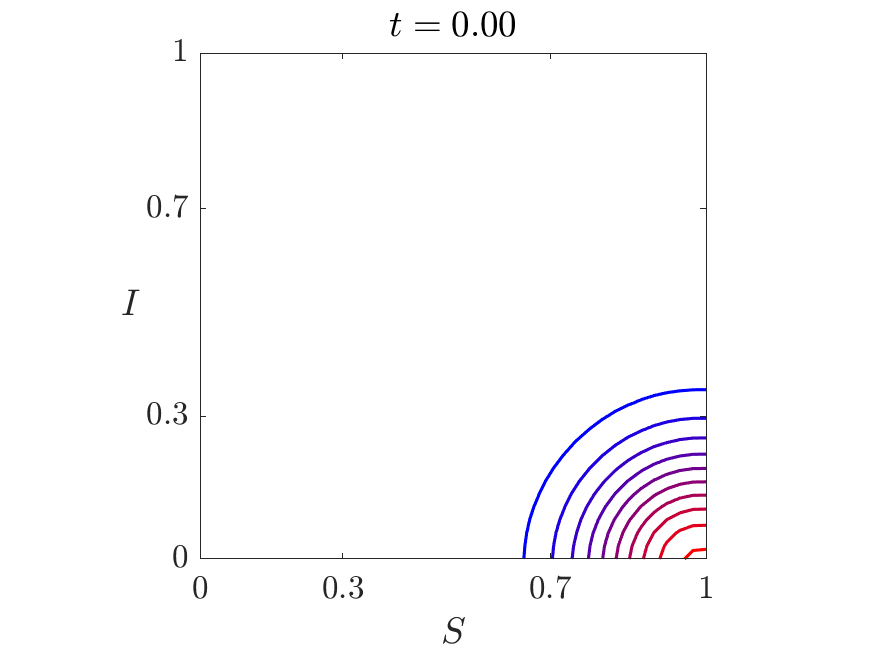} \, \includegraphics[width=0.3\textwidth,trim=55 0 75 0, clip]{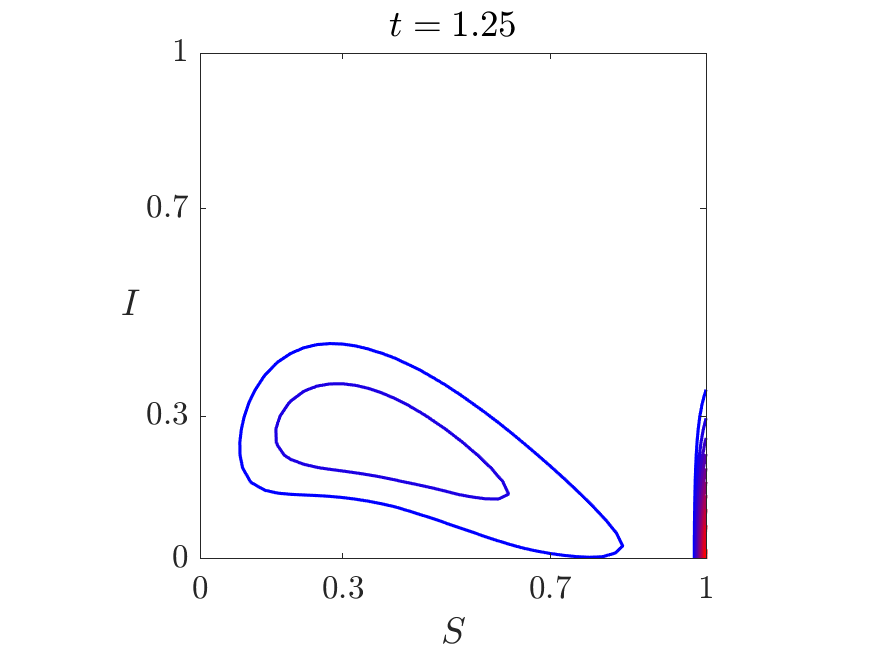} \,
\includegraphics[width=0.3\textwidth,trim=55 0 75 0, clip]{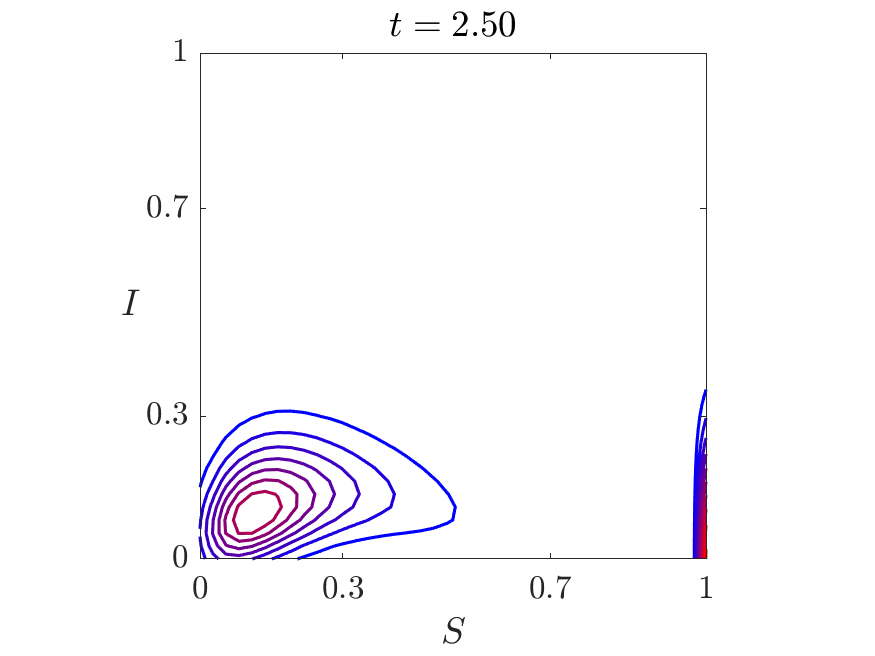} \\
\includegraphics[width=0.3\textwidth,trim=55 0 75 0, clip]{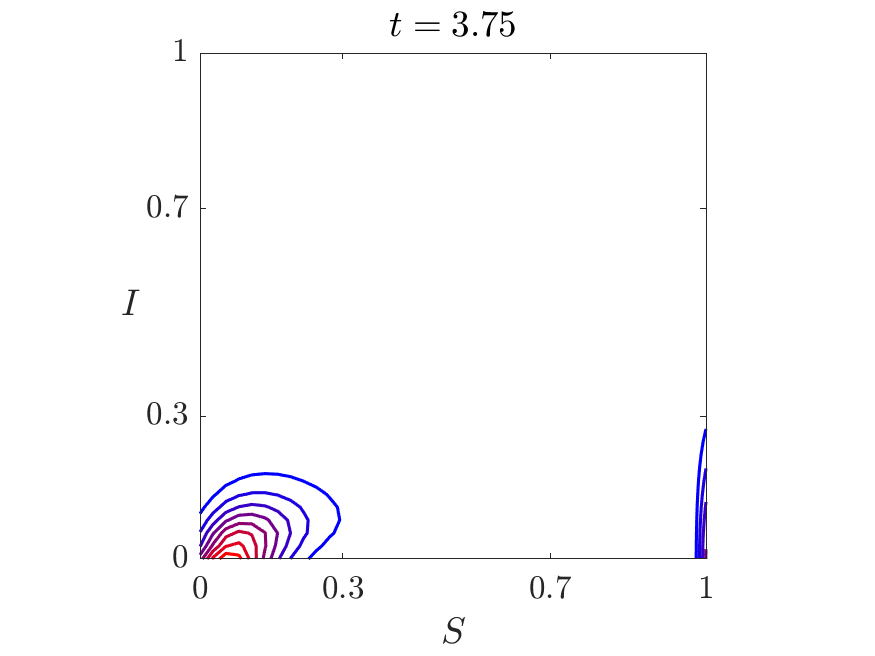} \,
\includegraphics[width=0.3\textwidth,trim=55 0 75 0, clip]{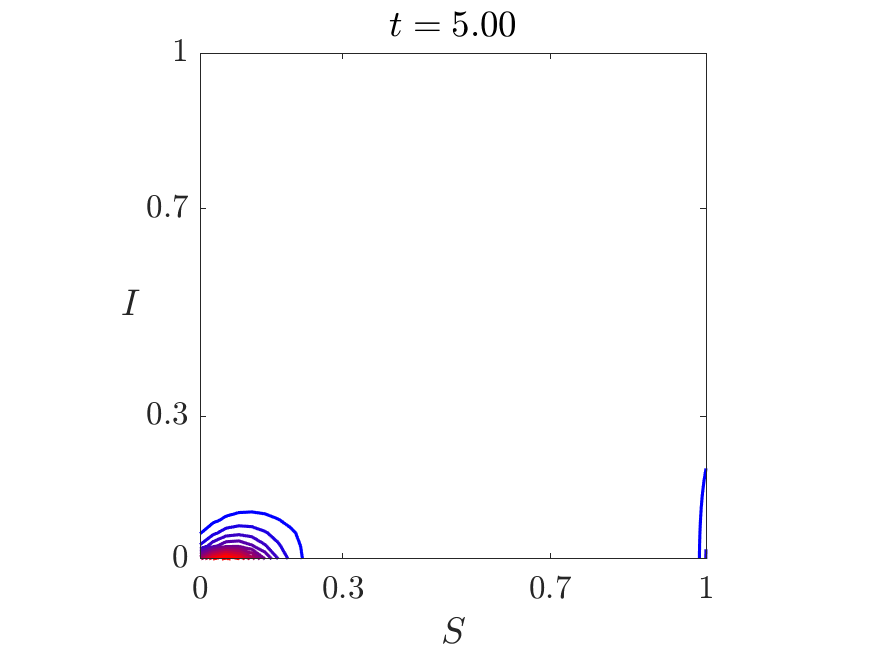} \,
\includegraphics[width=0.3\textwidth,trim=55 0 75 0, clip]{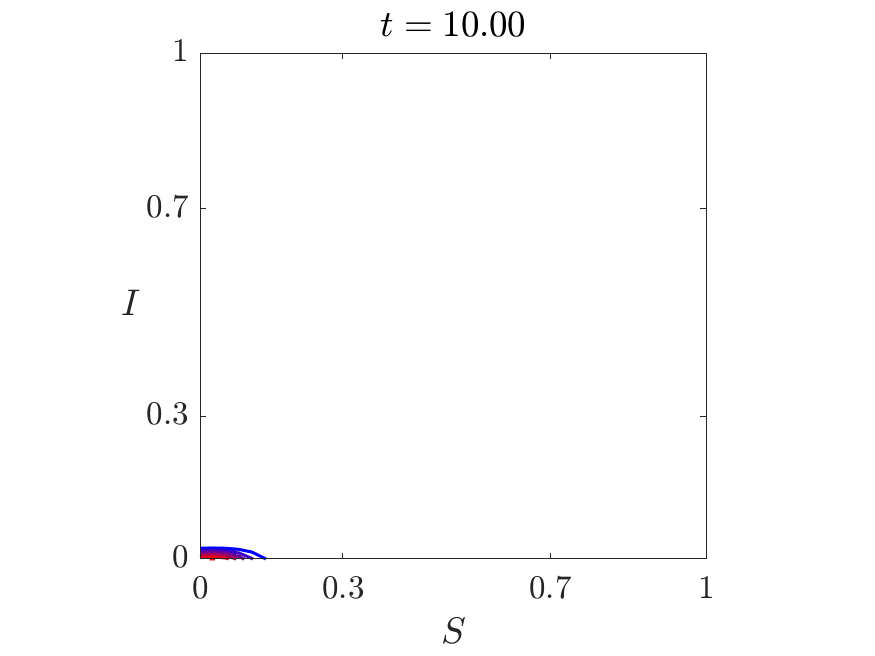} \,
\caption{Scenario 1: $G(x,t) = x_2 = I$, $K(x) \equiv 0$. In this scenario, the optimal control employs the maximal level of treatment efforts ($\eta(t)$) in the initial stage of infections, and as infections near their peak, these efforts are supplemented with modest use of NPIs ($\alpha(t)$) and vaccination ($v(t)$). }
\label{fig:scen1}
\end{figure}

\noindent {\bf Scenario 1.} In this scenario, we take $G(x,t) = 1.5I$ and $K(x) \equiv 0$, meaning there is no terminal cost and running cost is assessed purely based on the total number of infections (i.e. distribution functions $f(x,t)$ are penalized when they have more mass located in regions where $I$ is higher). Results of this simulation are contained in figure~\ref{fig:scen1}. In this case, the optimal control strategy is to implement maximum level of treatment efforts ($\eta(t)$) in the early stages of disease spread, and then supplement this with some use of NPIs ($\alpha(t)$ and vaccination ($v(t)$) when the infections near their peak. The net effect of this is to significantly slow down disease spread, as seen in the bottom panels which display the evolution of the distribution function $f(x,t)$ of possible states. Comparing the snapshots in figure~\ref{fig:scen1} with the corresponding snapshots for the uncontrolled simulation in figure~\ref{fig:uncontrolled}, we see that while the result at $t = 10$ is roughly the same, the mass in the controlled case travels more slowly from right to left, and is more concentrated on regions with higher $S$ and/or lower $I$ throughout (as can be seen at $t=1.25, 2.50, 3.75, 5.00$).\\

\noindent {\bf Scenario 2.} In this scenario, we still take terminal cost $K(x) \equiv 0$, but now choose running cost $G(x,t) = 1_{\{I\ge 0.15\}}(x)$ to be the indicator function of the set $\{I \ge 0.15\}$; that is, $$G(x,t) = \begin{cases}
    1, & I \ge 0.15, \\ 0, & I < 0.15.
\end{cases}$$ We use this to model something like a hospital capacity constraint, wherein no cost is assessed for infections until they reach a certain threshold beyond which hospitals are over-burdened. This is chosen to mirror some rhetoric from the early days of the COVID-19 pandemic, where policy-makers expressed concern over the excess burden to healthcare infrastructure and emphasized the need to ``flatten the curve." The results in figure~\ref{fig:scen2} demonstrate that the optimal controls for this cost functional very effectively accomplish this goal. In the plots, we have included the line $I = 0.15$ representing this threshold. Here the optimal control strategy is similar in the initial stages of the epidemic with maximal treatment efforts ($\eta(t)$). It appears that in this initial stage, it is infeasible to keep all the mass for $f(x,t)$ in the region $\{I<0.15\}$. However, subsequently, the optimal strategy involves maximal employment of vaccination ($v(t)$) and strong use of NPIs ($\alpha(t)$) to quickly drive the majority of the mass for $f(x,t)$ under the line $I = 0.15$. This is seen in the snapshots at times $t = 2.50,3.75,5.00$ where the contours have been significantly flattened when compared with Scenario 1. Likewise, the dynamics in the top left of figure~\ref{fig:scen2} exhibit precisely this ``flattening the curve" behavior, wherein the epidemic is prolonged, but the peak of infections is much smaller than in Scenario 1 or in the uncontrolled case. \\

\begin{figure}[t!]
\centering
\includegraphics[width=0.48\textwidth]{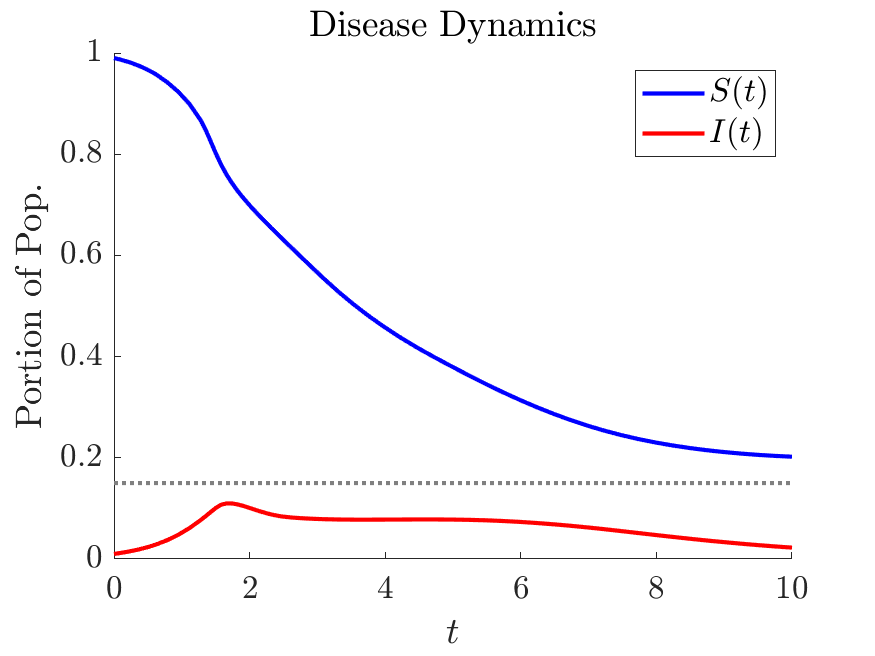} \,\,\, 
\includegraphics[width=0.48\textwidth]{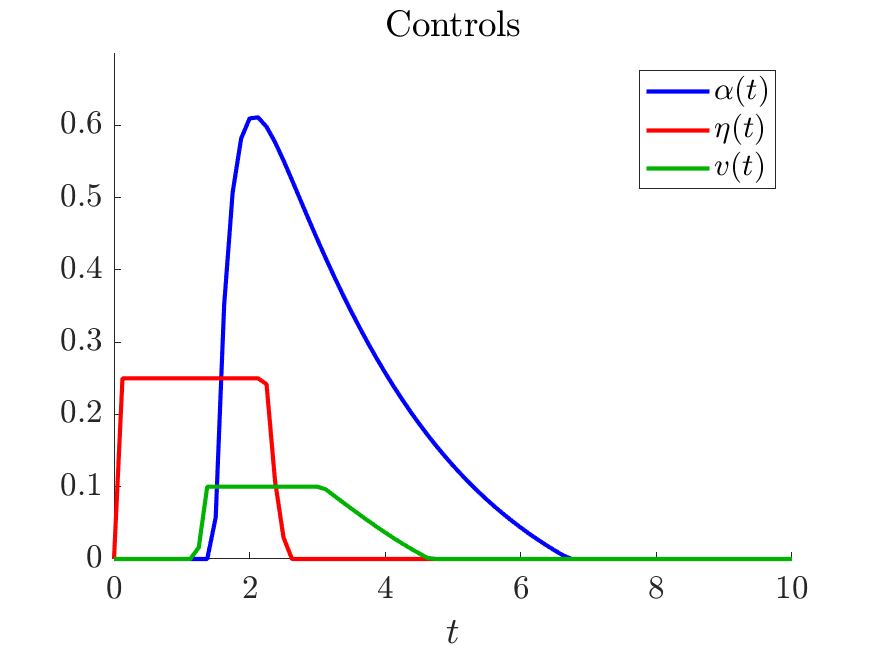}\\
\noindent\rule{0.9\textwidth}{1.5pt}
\includegraphics[width=0.3\textwidth,trim=55 0 75 0, clip]{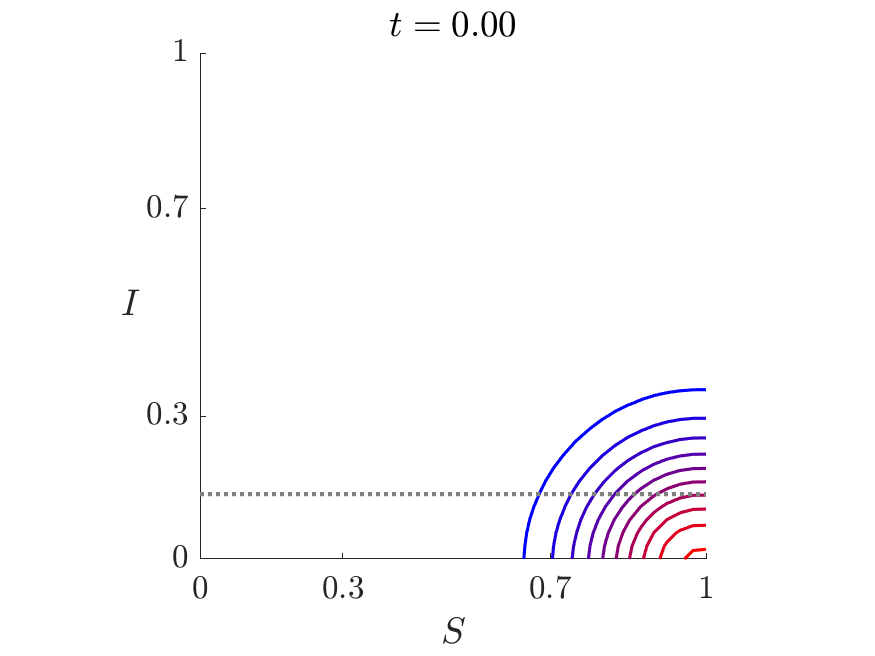} \, \includegraphics[width=0.3\textwidth,trim=55 0 75 0, clip]{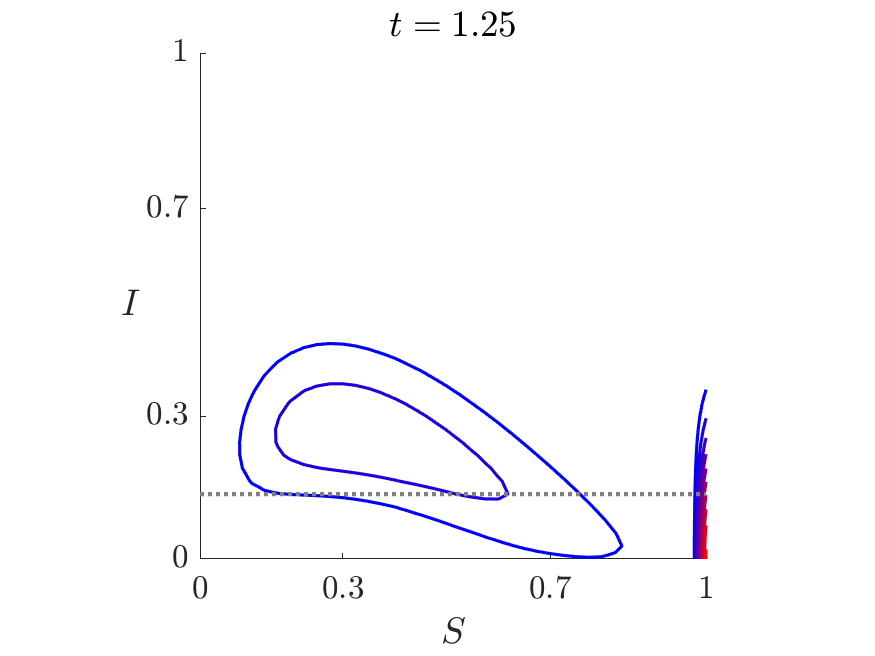} \,
\includegraphics[width=0.3\textwidth,trim=55 0 75 0, clip]{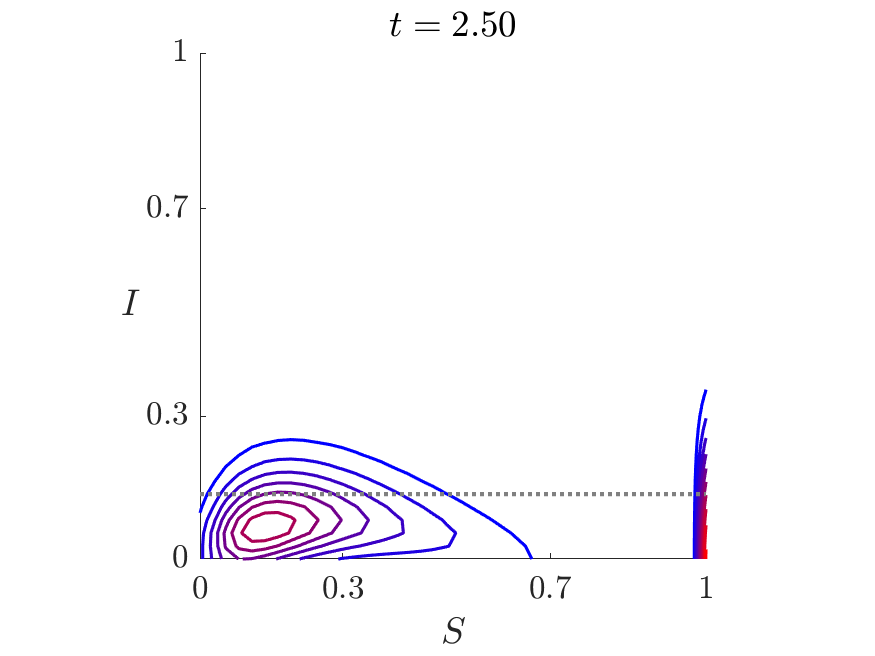} \\
\includegraphics[width=0.3\textwidth,trim=55 0 75 0, clip]{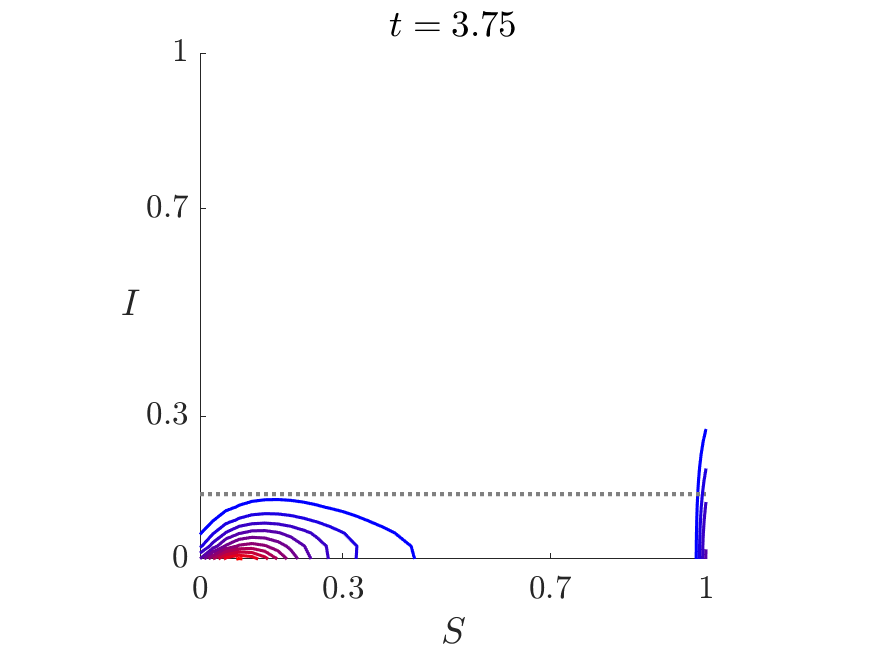} \,
\includegraphics[width=0.3\textwidth,trim=55 0 75 0, clip]{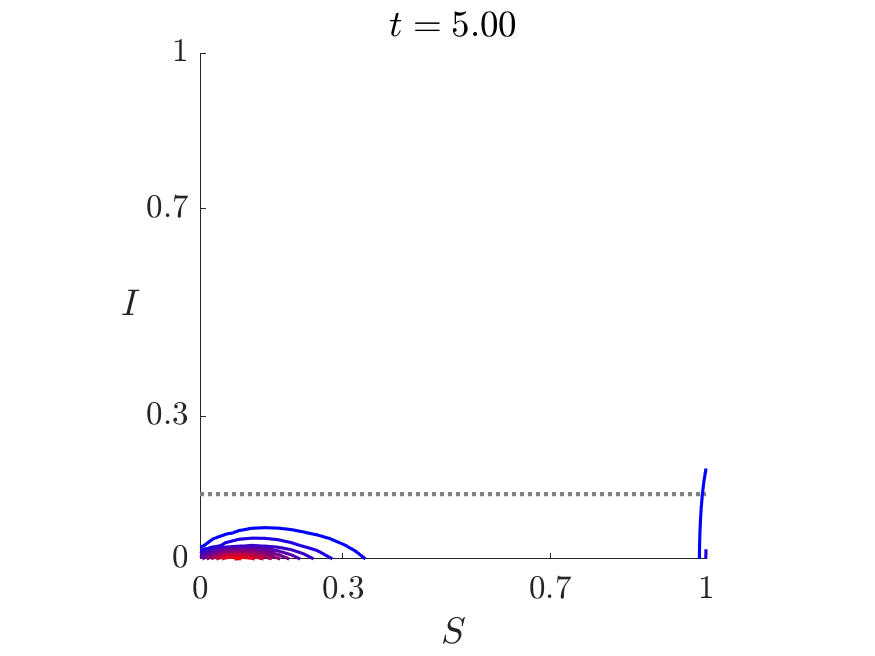} \,
\includegraphics[width=0.3\textwidth,trim=55 0 75 0, clip]{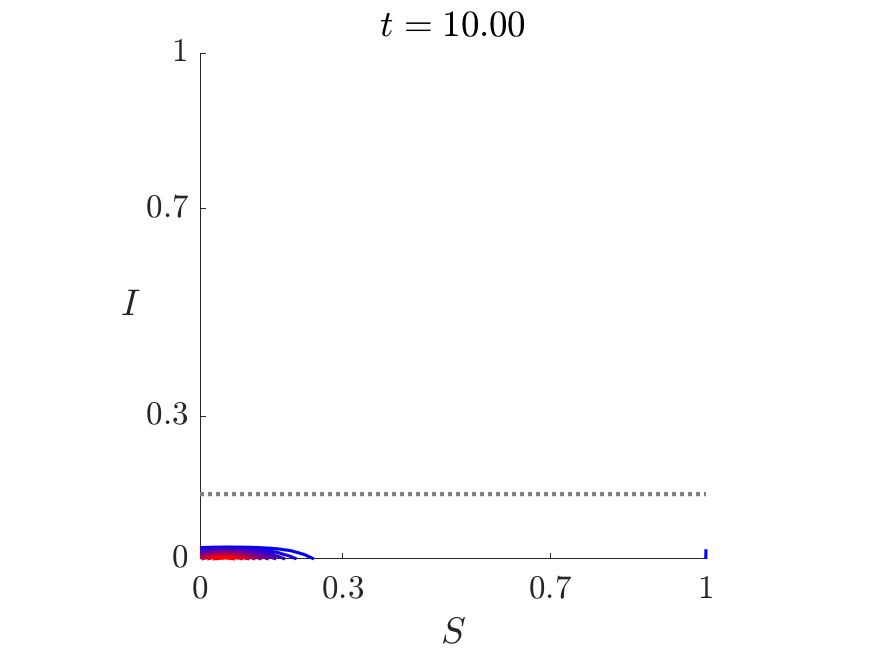} \,
\caption{Scenario 2: $G(x,t) = 1_{\{I\ge 0.15\}}(x)$, the indicator function of the set $\{I \ge 0.15\}$, $K(x) \equiv 0$. Here $0.15$ is imagined to be the infection threshold past which the healthcare system would be overburdened. In this scenario, the optimal control uses the maximal level of treatment efforts ($\eta(t)$) in the initial stage of infections, the maximum level of vaccination ($v(t)$) during the peak infections, and a substantially stronger employment of NPIs ($\alpha(t)$) than in Scenario 1. This has the effect of pushing the bulk of the mass for the distribution $f(x,t)$ below the line $I = 0.15$ (grey dotted line) much more quickly than in Scenario 1. With these controls, the dynamics (top left) exhibit a more prolonged but flattened epidemic.}
\label{fig:scen2}
\end{figure}

\begin{figure}[t!]
\centering
\includegraphics[width=0.48\textwidth]{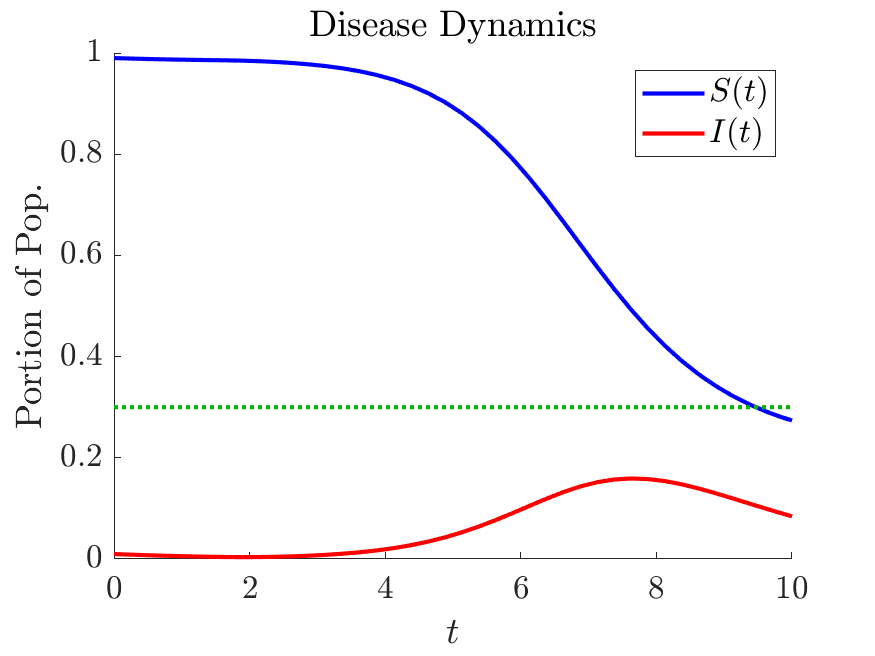} \,\,\, 
\includegraphics[width=0.48\textwidth]{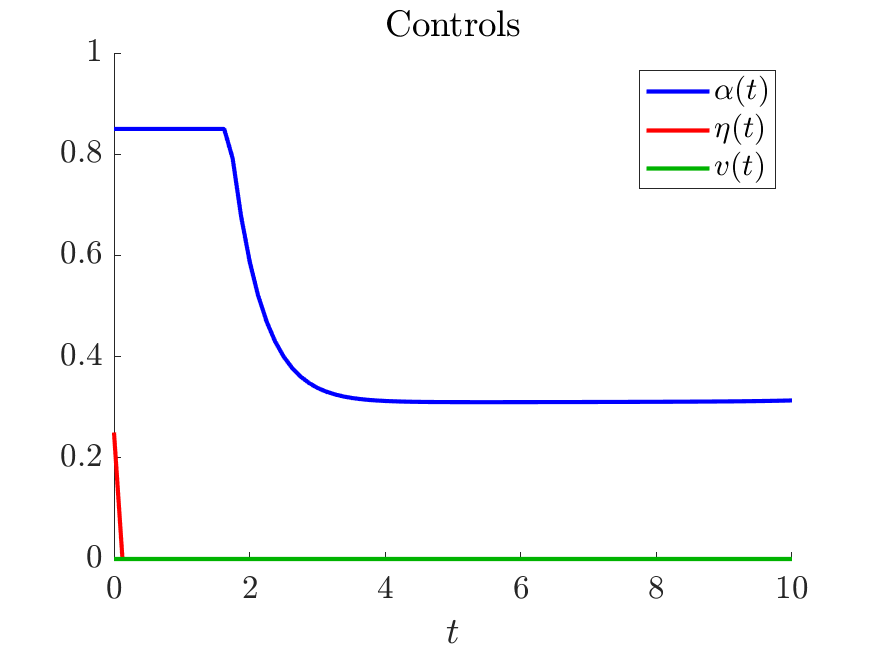}\\
\noindent\rule{0.9\textwidth}{1.5pt}
\includegraphics[width=0.3\textwidth,trim=55 0 75 0, clip]{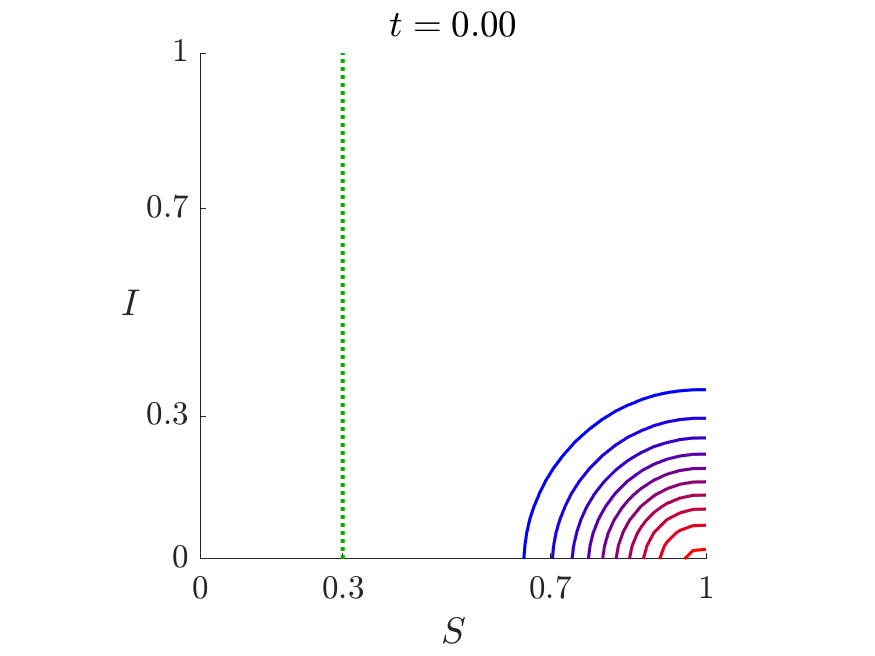} \, \includegraphics[width=0.3\textwidth,trim=55 0 75 0, clip]{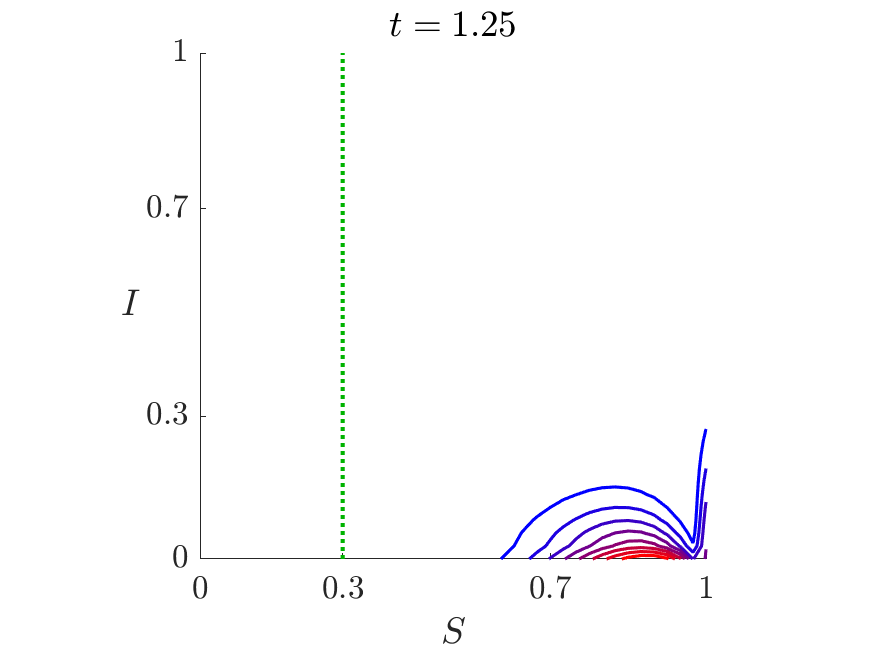} \,
\includegraphics[width=0.3\textwidth,trim=55 0 75 0, clip]{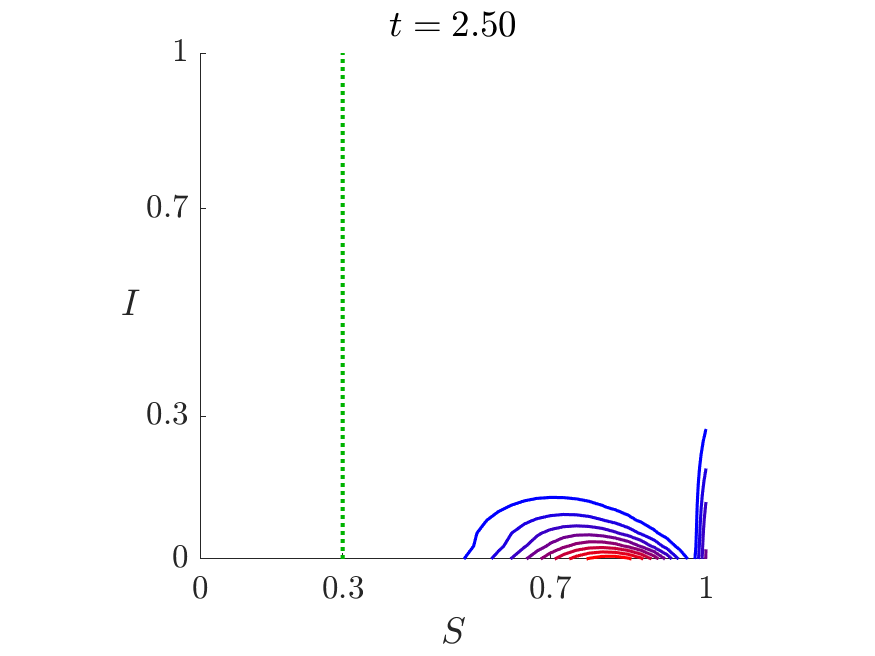} \\
\includegraphics[width=0.3\textwidth,trim=55 0 75 0, clip]{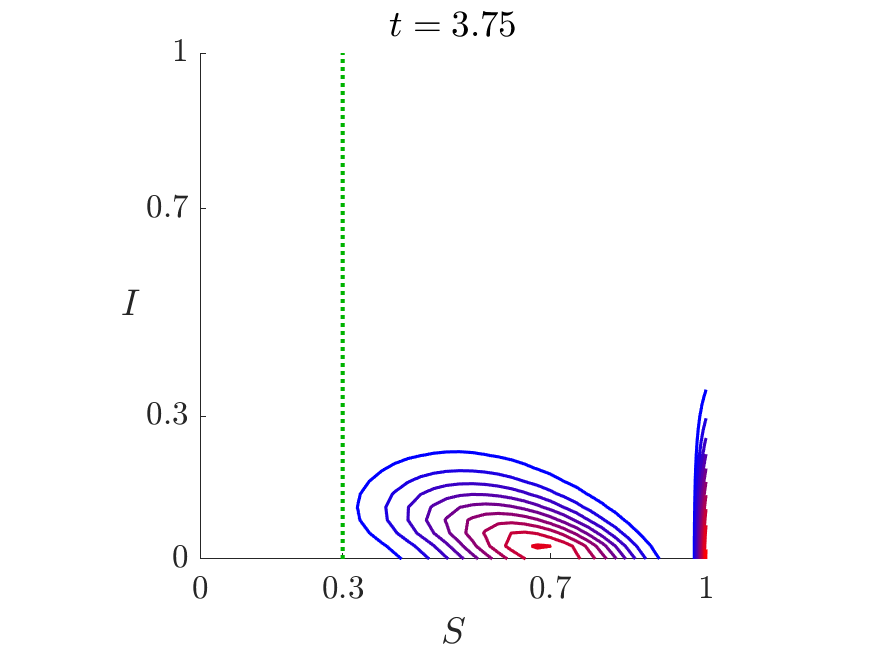} \,
\includegraphics[width=0.3\textwidth,trim=55 0 75 0, clip]{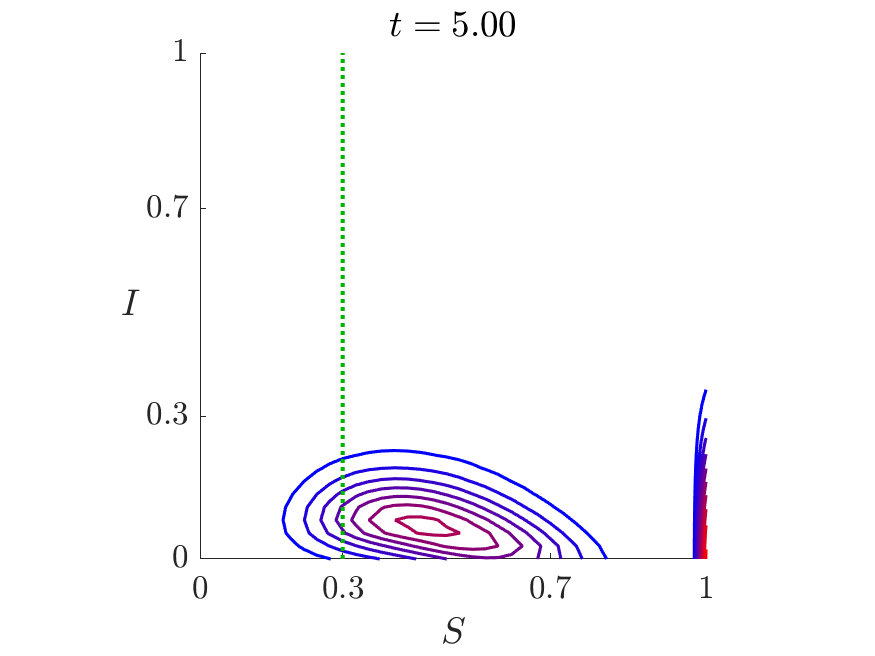} \,
\includegraphics[width=0.3\textwidth,trim=55 0 75 0, clip]{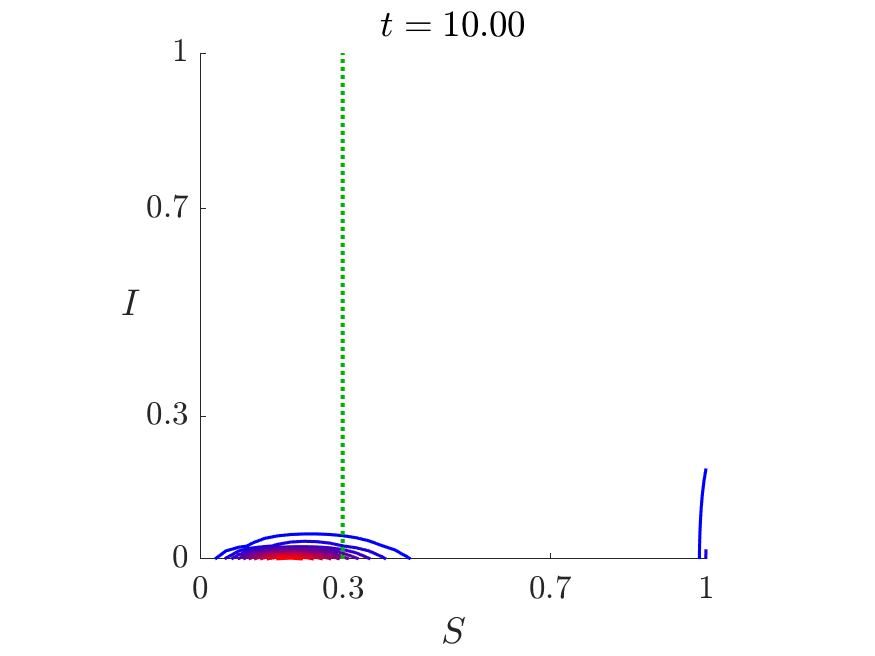} \,
\caption{Scenario 3: $G(x,t) \equiv 0$, $K(x) = -\max(S-0.3,0)$, $v(t)\equiv 0$. Here we simulate a scenario where vaccination is unavailable, and the only goal is to maintain a sufficiently large susceptible population by the end time. In this case, the optimal strategy entails very strong use of NPIs ($\alpha(t)$), and little to no use of treatment efforts ($\eta(t)$). The result is that the mass of the distribution function $f(x,t)$ moves much more slowly toward states with lower $S$ value, and ends with a nontrivial portion of the mass above the threshold $S = 0.3$ (green dotted line). In this case, the particular dynamics land below this threshold at time $t = 10$ (top left).}
\label{fig:scen3}
\end{figure}

\noindent {\bf Scenario 3.} Here we choose running cost $G(x,t) \equiv 0$ and terminal cost $K(x) = -\max(S-0.3,0)$, and we assume that vaccination is unavailable so that $v_{\text{max}} = 0.$ Here, with $K(x)$ negative, this is more akin to a profit that one can achieve by keeping the susceptible population at the end time above the threshold $0.3$. This scenario is meant to model the early stages of an epidemic wherein vaccines are not yet available, and the policy-maker would like to maintain a high level of susceptible population until the time that they are. The results are in figure~\ref{fig:scen3}. In this case, the only control that has meaningful effect is the NPIs ($\alpha(t)$). These are employed at their maximal strength at the outset, and then taper off but continue to be employed at a steady rate throughout the course of the epidemic. This will significantly stunt disease spread, whereupon treatment efforts ($\eta(t)$) are superfluous since there are essentially no new infected individuals who would benefit from an increased recovery rate. We have plotted the threshold $S = 0.3$ as green dotted lines in the plot of the dynamics and the snapshots of the density function $f(x,t).$ Here the effect of the controls is to significantly slow the drift of the mass of $f(x,t)$ toward states with lower $S$, such that, by $t = 10$, there is still a nontrivial amount of mass which lies in the region $S \ge 0.3$. That said, less than half the mass lies in this region, meaning that one is likely to end with $S < 0.3$, as is seen in the dynamics in the top left of figure~\ref{fig:scen3}. This is a key difference between the Fokker-Planck framework for control and the direct control of the ODE system \eqref{eq:controlledSIR}. In this case, one could say that the controls failed for this individual trajectory of the dynamics, in the sense that it lands below the threshold $S = 0.3$ at time $t = 10.$ However, cost is not being assessed for individual trajectories but rather for the distribution of possible trajectories. Here it seems that it was optimal to save some cost by arriving at a distribution wherein some small portion of the trajectories will indeed land above the threshold $S = 0.3$, while many will not.

\tred{It is important to emphasize that the resulting strategy is optimal only with respect to the chosen objective functional. While it preserves a larger susceptible population at the terminal time, it may not be optimal according to traditional epidemiological metrics such as minimizing cumulative infections or achieving rapid disease eradication. Such delay-oriented strategies may nevertheless be relevant when policymakers seek to postpone transmission while awaiting vaccine availability, improved treatments, or increased healthcare-system capacity.}

\section{Conclusions \& Future Work} \label{sec:conclusion}

In this manuscript, we have presented a framework for control of a stochastic compartmental model in mathematical epidemiology. The strategy entails controlling the associated Fokker-Planck equation, rather than controlling individual trajectories of the stochastic system. We formulate an optimal control problem involving a relatively general cost functional, prove existence of optimal controls, and derive first order optimality conditions which characterize the optimal controls. We describe the application of the sequential quadratic Hamiltonian method \cite{SQH} to our problem, and numerically resolve the optimal control maps for three different scenarios involving different policy-maker preferences. 

While the model presented is, admittedly, quite simple, this serves as a proof of concept that the Fokker-Planck control framework can be meaningfully applied to control of stochastic compartmental models for epidemiology, while allowing for uncertainty in initial data. As an immediate avenue of future work, the authors would like to apply a similar framework to more realistic models involving more heterogeneity, and in particular, to models like those developed by the first author which incorporate human behavior \cite{marcelo,OCbehavior2,PW3}. The barrier to doing so is that these models involves several additional compartments, and each compartment in the model becomes an additional ``spatial" dimension in the Fokker-Planck equation, meaning that grid-based numerical methods will be infeasible. Numerically solving high-dimensional PDE is an area of active research with popular approaches including tensor decomposition \cite{TD0,TD1,TD2} and pseudo-spectral methods \cite{PS1,PS2,PS3,PS4}. Development of similar schemes in the context of the Fokker-Planck equation would afford one the opportunity to apply our framework to a much broader class of epidemic models and thus greatly increase in the realism. \tred{Yet another avenue of future work would be to propose and analyze epidemiological models using Vlasov-McKean dynamics, wherein the stochastic differential equations describing populations dynamics also depend on the distribution function. This is done for basic SIS dynamics in \cite{VlasovMcKean}, but development of further models, inclusion of optimal control, and comparison with the Fokker-Planck control framework could prove very interesting. Finally, future work will also focus on integrating epidemiological surveillance data and data-assimilation techniques to enable disease-specific validation and policy evaluation. }

\section*{Acknowledgments} 

CP would like to thank Marcelo Bongarti for helpful discussions regarding the theoretical aspects of the optimal control problem. The work of SR was supported by the National Science Foundation grant numbers 2309491 and 2230790.

%\section*{Statements \& Declarations}
%
%\noindent{\bf Competing Interests:} the authors declare no competing interests.\\
%
%\noindent{\bf Data Availability:} no datasets were used in this work.

\bibliographystyle{siam}
\bibliography{ref}
\tred{
\section*{Appendix: Proofs of Some Lemmas from Section 3}

We include the proofs of some results that were not proven in the main body. \\

\noindent{\bf Lemma 3.1} \emph{The control to state map $\mS: U_{\text{ad}} \to \mathcal W$ is Lipschitz continuous.}

\begin{proof} For controls $u,\overline u \in U_{\text{ad}}$, we see that $z = \mathcal S(u)- \mathcal S(\overline u) = f-\overline f$ is a weak solution of \begin{equation} \label{eq:z} \partial_t z + \nabla \cdot(\bF(x,u)z) -\frac 1 2 \sum^2_{j=1} \frac{\partial^2}{\partial x_j^2}(\sigma_j(x,u)z) = \Phi(x,u,\overline u,t)\end{equation} where $$\Phi(x,u,\overline u,t) = \nabla \cdot\Big((\bF(x,\overline u)-\bF(x,u))\overline f(x,t)\Big) - \frac 1 2 \sum^2_{j=1} \frac{\partial^2}{\partial x_j^2}\Big((\sigma_j(x,\overline u)^2-\sigma(x,u)^2)\overline f(x,t)\Big),$$ along with $z(\cdot,0) = 0$ and a zero flux boundary condition. Since we have assumed that $\sigma_j(x,u)^2$ is Lipschitz in $u$, using that $\bF(x,u)$ is affine in $u$ and that $\|\overline f\|_{L^2(0,T;H^1(\Omega))}$ is uniformly bounded over choices of $\overline u \in U_{\text{ad}}$, we see $$\|\Phi(x,u,\overline u,t)\|_{L^2(Q)} \le C\|u-\overline u\|_{(L^2[0,T])^3}.$$ Then from classical parabolic estimates \cite[Ch. 9]{WYW}, we have that \begin{equation}
    \label{eq:zbound}\|\mathcal S(u) - \mathcal S(\overline u)\|_{L^2(0,T;H^1(\Omega))} = \|z\|_{L^2(0,T;H^1(\Omega))} \le C\|u-\overline u\|_{(L^2[0,T])^3},
\end{equation} demonstrating that control-to-state map is Lipschitz continuous.\end{proof}

We note that, once again invoking the compact embedding $\mathcal W \hookrightarrow\hookrightarrow L^2(Q)$, the norm on the left hand side of \eqref{eq:zbound} can be replaced by the $L^2(Q)$ norm if desired.\\

\noindent{\bf Lemma 3.2}
    \emph{For any two control maps $u, \overline u \in U_{\text{ad}}$ with respective solutions $f$ and $\overline f$ of \eqref{eq:FPcontrolled}, we have $$J(f,u) - J(\overline f,\overline u) = \int^T_0 \Big[H(t,\overline f(\cdot,t),q(\cdot,t),u(t)) - H(t,\overline f(\cdot,t),q(\cdot,t),\overline u(t)) \Big]dt$$ where $q$ is the solution of \eqref{eq:adjFPcontrolled} corresponding to control $u$.}

%\noindent{\bf Remark}. The proof relies on a type of weak formulation of \eqref{eq:FPcontrolled}. Specifically, if $f$ solves \eqref{eq:FPcontrolled}, then for any $\psi \in H^1(\Omega)$ and any $t \in (0,T)$, we have \begin{equation} \label{eq:weakFP} \int_\Omega \left(\partial_t f(x,t)\phi(x) + \nabla \cdot (\B F(x,u)f(x,t)) \psi(x)+ \frac{1}{2} \nabla (\B \sigma(x) f(x,t)) \cdot \nabla \psi(x) \right)dx = 0. \end{equation}

\begin{proof}
    Again, we suppress dependence on $(x,t)$ in the functions $f,\overline f,q$ for brevity, and compute \begin{equation} \label{eq:lem11}
    \begin{split}
        J(f,u) - J(\overline f,\overline u) &= \int^T_0 [\ell(u(t)) - \ell(\overline u(t))]dt + \int_\Omega K(x)[f-\overline f](\cdot, T) dx  + \int_0^T\int_\Omega G(x)[f-\overline f] dx\,dt \\
        &= \int^T_0 [\ell(u(t)) - \ell(\overline u(t))]dt + \int_\Omega K(x)[f-\overline f](\cdot, T) dx \\
        &\hspace{1cm}+\int^T_0\int_\Omega \Big(\partial_t q + {\B F(x,u)}\cdot {\nabla q} + \frac 1 2 \sum^2_{j=1} \sigma_j(x,u)^2 \frac{\partial^2q}{\partial x_j^2}\Big)[f-\overline f]dx\, dt\\
        &= \int^T_0 [\ell(u(t)) - \ell(\overline u(t))]dt \\ 
        &\hspace{1cm}+\int^T_0\int_\Omega \bigg(-[\partial_t f - \partial_t \overline f] - \nabla \cdot (\B F(x,u)[f-\overline f])\bigg) q\, dx dt \\ &\hspace{1cm} + \int_0^T\int_\Omega\frac 1 2\sum^2_{j=1} \frac{\partial^2}{\partial x_j^2}\Big(\bS_j(x,u)^2[f-\overline f]\Big)q \, dxdt,
    \end{split}   
    \end{equation} where, in the last line, we have integrated by parts and canceled boundary terms. Now since $f$ satisfies \eqref{eq:FPcontrolled} with control $u$, we see that all terms with $f$ vanish. Likewise, we replace $\partial_t \overline f$ using \eqref{eq:FPcontrolled} with control $\overline u$. Doing so, we see  \begin{equation} \label{eq:lem12}
    \begin{split}
        J(f,u) - J(\overline f,\overline u) &= \int^T_0 [\ell(u(t)) - \ell(\overline u(t))]dt \\
        &\hspace{1cm} + \int_0^T\int_\Omega  \nabla \cdot\Big([\bF(x,u)-\bF(x,\overline u)]\overline f\Big) q \, dx dt \\ &\hspace{1cm}  -\int^T_0 \int_\Omega \frac 1 2\sum^2_{j=1} \frac{\partial^2}{\partial x_j^2}\Big((\sigma_j(x,u)^2 - \sigma(x,\overline u)^2\overline f\Big)\bigg)q \, dxdt\\
        &= \int^T_0 \ell(u(t)) - \bigg(\int_\Omega  \overline f \bF(x,u)\cdot \nabla q - \frac 1 2 \sum^2_{j=1} \overline f \sigma_j(x,u)^2 \frac{\partial^2 q}{\partial  x_j^2} dx \bigg)dx  \\
        &\hspace{1cm}-\int^T_0 \ell(\overline u(t)) - \bigg(\int_\Omega  \overline f \bF(x,\overline u)\cdot \nabla q - \frac 1 2 \sum^2_{j=1} \overline f \sigma_j(x,\overline u)^2 \frac{\partial^2 q}{\partial  x_j^2} dx \bigg)dx\\
        &= \int^T_0 \Big[H(t,\overline f(\cdot,t),q(\cdot,t),u(t)) - H(t,\overline f(\cdot,t),q(\cdot,t),\overline u(t)) \Big]dt
    \end{split}   
    \end{equation} as desired.
\end{proof}}

\bigskip\noindent 
Christian Parkinson, Department of Mathematics \& Department of Computational Mathematics, Science and Engineering, Michigan State University, East Lansing, MI, USA;
e-mail: \url{chparkin@msu.edu}

\medskip\noindent 
Souvik Roy, Department of Mathematics, University of Texas at Arlington, Arlington, TX, USA;
 e-mail: \url{souvik.roy@uta.edu}

\end{document}